\documentclass[a4paper,11pt]{amsart}
\usepackage{cases}
\usepackage{bm}
\usepackage{amsfonts,amsmath,amssymb,amscd,bbm,amsthm,mathrsfs,dsfont}
\usepackage{hyperref}
\hypersetup{
	colorlinks=true,
	linkcolor=blue,
	filecolor=magenta,
	urlcolor=cyan,
	citecolor=red
}
\usepackage{cleveref}
\usepackage[all]{xy}
\SelectTips{cm}{}

\usepackage{graphicx}
\usepackage{psfrag}

\usepackage{tikz}
\usepackage{tikz-cd}
\usepackage{arydshln}

\usepackage{amssymb}
\usetikzlibrary{decorations.pathreplacing}
\usetikzlibrary{matrix,arrows}
\usepackage{enumitem}

\allowdisplaybreaks[3]

\newcommand{\ie}{{\em i.e.}\ }

\newtheorem{theorem}{Theorem}[section]
\newtheorem{lemma}[theorem]{Lemma}
\newtheorem{definition}[theorem]{Definition}
\newtheorem{proposition}[theorem]{Proposition}
\newtheorem{corollary}[theorem]{Corollary}

\newtheorem{example}[theorem]{Example}
\newtheorem{remark}[theorem]{Remark}

\newcommand{\reminder}[1]{}
\newcommand{\opname}[1]{\operatorname{\mathsf{#1}}}

\newcommand{\ind}{\opname{ind}\nolimits}

\newcommand{\op}{^{op}}

\newcommand{\lcm}{\mathrm{lcm}}

\newcommand{\Ga}{\Gamma}

\newcommand{\trdeg}{\mathrm{trdeg}}

\newcommand{\Ext}{\opname{Ext}}
\newcommand{\ET}{\opname{ET}}
\newcommand{\Z}{\opname{Z}}
\newcommand{\I}{\opname{I}}
\newcommand{\Lt}{\opname{Lt}}
\newcommand{\Lm}{\opname{Lm}}

\newcommand{\Aut}{\opname{Aut}}

\newcommand{\ca}{{\mathcal A}}

\renewcommand{\tilde}[1]{\widetilde{#1}}

\makeatletter
\@addtoreset{equation}{section}
\makeatother

\begin{document}

\title{Generalized frieze varieties and Gr\"obner bases}
\author{Siyang Liu}
\address{School of Mathematics, Hangzhou Normal university,  Hangzhou, China}
\email{siyangliu@hznu.edu.cn}

\begin{abstract}
	We study  properties of generalized frieze varieties for quivers associated to cluster automorphisms. Special cases include acyclic quivers with Coxeter automorphisms and quivers with Cluster DT automorphisms. We prove that the generalized frieze variety $X$ of  an affine quiver with the Coxeter automorphism  is either a finite set of points or a union of finitely many rational  curves.  In particular, if $\dim X=1$, the genus of each irreducible component of $X$  is zero. We also propose an algorithm to obtain the defining polynomials for each irreducible component of the generalized frieze variety  $X$ of an affine quiver. Furthermore, we give the Gr\"obner basis with respect to a given monomial order for each irreducible component  of frieze varieties of affine quivers with  given orientations,  and show that each component is a smooth rational curve.
\end{abstract}

\keywords{Generalized frieze variety, Cluster automorphism, Affine quiver, Rational curve, Gr{\"o}bner basis.}

\subjclass[2010]{13F60,16G20,13P10}

\maketitle

\tableofcontents

\section{Introduction}	
The {\em frieze pattern} associated to an acyclic quiver $Q$ is a sequence of  positive rational  vectors  \[\mathbf a_t =  (a_{i,t}, i\in Q_0),\quad t\in \mathbb N,\]
such that  these vectors are uniquely determined by the initial conditions $a_{i,0} = 1$ for $i\in Q_0$ together with the  relations
\[       a_{i,t+1} {a_{i,t}}  =\prod_{j\rightarrow i} a_{j,t}\cdot\prod_{i\rightarrow k}a_{k,t+1} +1.       \]

The frieze patterns for quivers of type $\mathbb A$ are precisely  the  Convey-Coxeter's friezes \cite{convey1973frieze}.  Caldero and Chapoton found a connection between cluster algebras and friezes  in \cite{caldero2006hall} which inspired more authors to study friezes via cluster theory. More specifically, the frieze pattern can be obtained from a sequence of clusters, which are  given via applying the so-called  Coxeter mutation to the initial cluster  repeatedly,  by specializing 
 initial cluster variables to one.  Here the {\em Coxeter mutation associated to an acyclic quiver}  is the composition of the mutation sequence given by an admissible sequence of sinks of the quiver.

For each vertex $i\in Q_0$, the sequence $(a_{i,t})_{t\in \mathbb N}$ is called a {\em frieze sequence associated to the vertex $i$ of $Q$}.  Assem, Reutenauer and Smith conjectured the frieze sequences associated to an acyclic quiver $Q$ satisfy linear recurrences if and only if $Q$ is a  Dynkin  quiver or an  affine  quiver, and they proved the "only if" part is always true and the converse is true when $Q$ is  of affine types $\tilde{\mathbb A}$ and $\tilde{\mathbb D}$ in \cite{assem2010frieze}. The "if" part of the conjecture was proved in full generality by Keller  and  Scherotzke in \cite{keller2011linear} using (generalized) exchange triangles in cluster categories.  
Their work provides a new perspective for an acyclic quiver to be finite, tame or wild in terms of rationality of frieze sequences.

Another characterization of the  finite-tame-wild trichotomy for acyclic quivers via friezes was obtained by Lee, Li,  Mills, Schiﬄer and Seceleanu in \cite{lee2020frieze}. More precisely,  they viewed  the frieze pattern  associated to each acyclic quiver $Q$ as a subset of  the affine space $\mathbb C^n$ and  defined the {\em frieze variety} $X(Q)$ to be the  Zariski closure of the set of vectors in the corresponding frieze pattern, and they showed that $\dim X(Q) = 0$  when $Q$ is of Dynkin type, $\dim X(Q) = 1$ when $Q$ is of affine type, and  $\dim X(Q) >1$ when $Q$ is of wild type.  It turns out to be interesting to study   properties of frieze varieties, and there are also some questions on frieze varieties to be answered.

Our main motivation in this paper is to answer a conjecture and some problems  proposed  by  Lee, Li,  Mills, Schiﬄer and Seceleanu in \cite{lee2020frieze}  which are listed as follows.
\begin{itemize}[itemsep=2pt,topsep=0pt,parsep=0pt]
	\item Problem:  Study  the friezes obtained by changing the initial conditions.
	\item Problem:  How to define the frieze variety for any quiver which is not necessarily acyclic but mutation equivalent to an acyclic quiver and study how does it behave under mutation.
   \item Conjecture:  In the affine case, each irreducible component of the frieze variety is a smooth curve of genus zero.
\end{itemize}

For the first problem, Igusa and Schiffler defined the  {\em generalized frieze variety} by letting the initial condition be more general (see Definition  \ref{defgfv}), and they also proved the  Coxeter mutation cyclically permutes irreducible components of a generalized frieze variety (see Lemma \ref{isinvariant}) in \cite{igusa2021frieze}.  

For the second problem,  recall that  the frieze pattern is obtained by applying the Coxeter  mutation repeatedly to the initial cluster, while  the Coxeter mutation arises from an admissible sequence of sinks which forms a reddening sequence (indeed a maximal green sequence) of the given acyclic quiver.  
 For reddening sequences and their relation with DT invariants,  we refer to  \cite{keller2011dilog,keller2017dt}. Since reddening sequences preserve quivers, each reddening sequence induces a cluster automorphism which is called a {\em cluster DT automorphism}\footnote{ We call it cluster DT automorphism because the isomorphism on  cluster varieties induced by the cluster automorphism on upper cluster algebras is  called the  cluster DT transformation and studied by Goncharov and Shen in \cite{goncharov2018dt}.}.  In particular,  we call the cluster DT automorphism induced by the Coxeter mutation  a {\em Coxeter automorphism}.  In conclusion,  we have the following relations.
\[ \{ \text{Coxeter automorphisms}\}\subset \{\text{Cluster DT automorphisms}\}\subset \{\text{Cluster automorphisms}\}.        \]

In this paper, we define the generalized frieze variety for any quiver $Q$ by replacing the Coxeter mutation with a cluster automorphism, see Definition \ref{defgfv}. This definition naturally  extends the definition of (generalized) frieze varieties considered in \cite{lee2020frieze,igusa2021frieze}, and we then study their properties, such as their irreducibility and they behave under mutations and folding theory.

The group of cluster automorphisms is  called a {\em cluster automorphism group} in \cite{assem2012automorphism}.  By [Theorem 3.8, 3.11, \cite{assem2012automorphism}],  the  cluster automorphism groups for acyclic quivers depend largely on the Coxeter automorphisms.  See the [Table 1,  \cite{assem2012automorphism}]  for a full list of the cluster  automorphism groups of affine quivers.  In this paper,  we  are mainly concerned about generalized frieze varieties associated to the Coxeter automorphism for acyclic quivers.   In particular, for affine quivers, we obtain the following results.

\begin{theorem}
Let $Q$ be an affine quiver and $f_c$ the Coxeter automorphism.  Let $X(Q,f_c,\mathbf a)$ denote  the generalized frieze variety for a general specialization $\mathbf a\in (\mathbb C^*)^n$. Then  $X(Q,f_c,\mathbf a)$ is  either  a   finite set or  a union of  finitely many rational curves.
In particular, if $X(Q,f_c,\mathbf a)$ has dimension one,  the genus of each irreducible component is zero.
\end{theorem}

  Our proof is uniform for generalized frieze varieties of affine quivers  and  extends the results for affine quivers in \cite{lee2020frieze}  which they shall let $\mathbf a=\mathbf 1$ and check linear recurrences for exceptional types by a computer.  We then provide a computational algorithm  for obtaining the defining ideals of generalized frieze varieties of affine types.  Moreover this algorithm allows us to compute Gr\"obner bases of generalized frieze varieties for affine quivers with arbitrary general specializations and arbitrary orientations of Euclidean diagrams.

Furthermore, for affine quivers with orientations given in Section \ref{chapteraffine},
we compute the Gr\"obner bases with respect to a given lexicographical order for ideals of irreducible components of frieze varieties, and we therefore obtain that each component is smooth. 
In particular, for frieze varieties of type $\tilde{\mathbb A}_{p,q}$, we will see that each component is  isomorphic to either $\mathbb C$  or a  plane conic, here by a plane conic,  we mean  a  curve in $\mathbb C^2$ determined by a polynomial with maximal degree   $2$.  For frieze varieties of type $\tilde{\mathbb D}_{n}$, each component is  isomorphic to a  plane conic. To conclude, we have the following result.

\begin{theorem}
Let $Q$ be an affine quiver with orientation given in Section \ref{chapteraffine}, then each irreducible component of 
the frieze  variety $X(Q,f_c,\mathbf 1)$ is smooth.
\end{theorem}

The paper is organized as follows.  In Section \ref{preli}, we recall basic facts including cluster theory, Gr\"obner bases and linear recurrences, and we prove some elementary properties about affine varieties we will need later.  In Section \ref{fv}, we define generalized frieze varieties for quivers associated to the cluster automorphism and we explain they extend definitions in \cite{lee2020frieze,igusa2021frieze}.  We also study properties of generalized frieze varieties, such as they behave under mutations or folding techniques. We prove generalized frieze variety of an affine quiver associated to Coxeter mutation is either a finite subset or  a finite union of rational curves in Section \ref{chapteraffine}.  Finally in Section \ref{groaff}, we study  Gr\"obner bases  for frieze varieties of affine quivers.

\section{Preliminaries}\label{preli}
\subsection{Cluster algebras and cluster categories}
In this paper, a {\em quiver} is a finite connected directed graph without loops or oriented $2$-cycles.  For a quiver $Q$, we usually denote the set of vertices by $Q_0$ and the set of arrows by $Q_1$.
A  {\em  seed}  is a pair $(\mathbf{x}, Q)$, where
\begin{enumerate}
	\item[-] $Q$ is a quiver with $n$ vertices labeled by  $1, 2, \dots, n$; 
	\item[-] $\mathbf{x} =  (x_1, \dots, x_n)$  is  a  set of indeterminates over $\mathbb{Q}$ such that the {\em ambient field} $\mathcal{F}=\mathbb{Q}(x_1, \dots, x_n)$ is a purely transcendental field extension of $\mathbb{Q}$, and $\mathbf{x}$ is called a {\em cluster},  $x_1,\dots,x_n$ are called 
	{\em  cluster   variables}.
\end{enumerate}

For $k\in [1,n]$, define another pair $(\mathbf{x'}, Q') = \mu_k(\mathbf{x},  Q)$ which is called the {\em mutation} of $(\mathbf{x}, Q)$
at $k$ and obtained by the following rules:
\begin{enumerate}
	\item[-]  the quiver $Q'$ is obtained from $Q$ by the following operations:
	\begin{enumerate}
		\item[(i)] for every $2$-path $i\rightarrow k\rightarrow j$, add a new arrow $i\rightarrow j$;
		\item[(ii)] reverse all arrows incident with $k$;
		\item[(iii)] delete a maximal collection of $2$-cycles from those created in (i).
	\end{enumerate}
	\item[-] $\mathbf{x'} = (x'_1, x'_2,\dots, x'_n)$ is given by \[x'_i =\begin{cases} x_i, &\text{$i\neq k$},\\  
		\frac{   \prod\limits_{i\rightarrow k}x_i +    \prod\limits_{k\rightarrow j}x_j         }{x_k} ,&\text{$i=k$.} \end{cases}\]
\end{enumerate}

Let $\mathcal{X}$ be the union of clusters which can be obtained from $(\mathbf x, Q)$ by a finite sequence of mutations. The {\em cluster algebra} $\mathcal{A} = \mathcal{A}(\mathbf x, Q)$ is the $\mathbb{Z}$-subalgebra of $\mathbb{Q}(x_1,x_2,\dots,x_n)$ generated by all the cluster variables in $\mathcal{X}$. 

The Laurent phenomenon and positive property are the most important properties of cluster algebras.

\begin{theorem}[ Laurent phenomenon and positivity \cite{fomin2002fundations, lee2015positivity}]\label{laurentpositivity}
	Let $\mathcal A = \mathcal A(\mathbf x, Q)$ be a cluster algebra, and $x\in \mathcal X$  any cluster variable of $\mathcal A$. Then we have 
	\[\mathcal A\subset \mathbb Z[x_1^{\pm 1}, x_2^{\pm 1}, \dots, x_n^{\pm 1}],\,\,\,\text{and}\,\; x\in \mathbb Z_{\geq 0}[x_1^{\pm 1}, x_2^{\pm 1}, \dots, x_n^{\pm 1}].\]
\end{theorem}

For any cluster algebra, it is proved in \cite{shapiro2008exchange} that a seed $(\mathbf x, Q)$ is uniquely determined by its cluster $\mathbf x$, we will  use  $Q(\mathbf x)$ to denote the quiver associated to the cluster $\mathbf x$ and we also use $\mu_x$ to denote the mutation at the vertex of $Q(\mathbf x)$ corresponding to the cluster variable $x\in \mathbf x$.

Let $\mathcal{A} =\mathcal A(\mathbf x, Q)$ be a cluster algebra. A  $\mathbb{Z}$-algebra  automorphism $f: \mathcal{A} \rightarrow \mathcal{A}$  is called  a  {\em cluster automorphism} if there is another seed $(\mathbf x', Q')$ of $\mathcal A$ such that 
\begin{enumerate}
	\item[(CA1)]  $f(\mathbf x) = \mathbf x'$;
	\item[(CA2)]  $f(\mu_{x}(\mathbf x))  =  \mu_{f(x)}(\mathbf x')$ for every $x\in \mathbf x$.
\end{enumerate}

 As pointed out in \cite{assem2012automorphism}, each cluster automorphism is uniquely determined by its value on one cluster and naturally induces an automorphism of the ambient field $\mathcal F$. In general, the converse is not true. However, the following assertion is clear.  

\begin{lemma}[\cite{cao2019automorphism}]\label{cllp}
	Let $\mathcal A=\mathcal{A}({\bf x}, Q)\subseteq \mathcal F$ be a cluster algebra, and let $f$ be an automorphism of the ambient field $\mathcal F$. If there exists another seed $(\mathbf{z},Q')$ of $\mathcal{A}$ such that $f({\bf x})={\bf z}$ and $f(\mu_x({\bf x}))=\mu_{f(x)}({\bf z})$ for any $x\in{\bf x}$. Then $f$ restricts to a  cluster automorphism of $\mathcal{A}$. 
\end{lemma}

Suppose that $f$ is a $\mathbb Z$-algebra automorphism  satisfying the condition (CA1)  for a seed  $(\mathbf x,Q)$, that is, $f(\mathbf x)$ is also a cluster, we introduce the following two conditions:

(CA3)\;\;  $Q(f(\mathbf x)) \cong Q$;\quad   (CA4)\;\;  $Q(f(\mathbf x)) \cong Q^{\op}$.

Cluster automorphisms have the following equivalent characterizations.

\begin{proposition}[\cite{assem2012automorphism,cao2019automorphism}]
	Let  $f: \mathcal{A} \rightarrow \mathcal{A}$  be a $\mathbb Z$-algebra automorphism. Then the following statements are equivalent:
	\begin{enumerate}
	\item\ $f$ is a cluster automorphism;
	\item\ $f$ satisfies the condition  {\em (CA1)} for a seed;
	\item\   $f$ satisfies  the condition {\em (CA1)} for every seed;
	\item\  $f$ satisfies conditions {\em (CA1) $\&$ (CA3)}  for a seed or $f$ satisfies conditions {\em (CA1) $\&$ (CA4)}  for a seed $(\mathbf x, Q)$
	\item\  $f$ satisfies conditions {\em (CA1) $\&$ (CA3)} for every seed or $f$ satisfies conditions {\em (CA1) $\&$ (CA4)}  for every seed.  
 	\end{enumerate}
	\end{proposition}

 We call a cluster automorphism $f$ a {\em direct cluster automorphism} if the condition (CA3) is satisfied and an {\em inverse cluster automorphism} if the condition (CA4) is satisfied.  By [Proposition 2.15 $\&$ 2.17, \cite{assem2012automorphism}],  if $\sigma\in S_n$ is an automorphism (or an anti-automorphism)  of $Q$, then 
$\sigma$ induces a  direct (or an inverse) cluster automorphism  $f_{\sigma}$ defined  by 
 \[f_{\sigma}(x_i) := x_{\sigma(i)},\quad i\in Q_0.\] 
It is clear that the cluster automorphism $f_{\sigma}$  satisfies that
\[     f_{\sigma} (\mu_{i_p}\dots\mu_{i_1}(\mathbf x))  =   \mu_{\sigma(i_p)}\dots\mu_{\sigma(i_1)}(f_{\sigma}(\mathbf x))  \]
for each mutation sequence $\mu_{i_p}\dots\mu_{i_1}$.

The cluster category, introduced in \cite{buan2006clustercat}, provides a framework
 of representation theory of quivers to study cluster algebras. We recall some basic facts on cluster categories and the Caldero-Chapoton map, and we  refer to \cite{keller2010quiverrep} for further details.
 
 Let $Q$ be an {\em acyclic quiver}, \ie, a quiver without oriented cycles, and let $\mathcal D_Q=\mathcal{D}^b({\mathbb C}Q^{op})$ denote the bounded derived category of finitely generated $\mathbb CQ^{op}$-modules\footnote{Here we use the opposite quiver to make our notations adjust to that in \cite{keller2011linear}.} with the shift functor $\Sigma$. Denote by $\tau$ the  AR-translation  of 
$\mathcal D_Q$. Then the {\em cluster category} $\mathcal C_Q$ is defined to be the orbit category 
\[\mathcal  C_Q := \mathcal D_Q / \langle \tau^{-1} \Sigma\rangle,\] which is a   Krull-Schmidt  $2$-Calabi-Yau  triangulated category with isomorphism classes of indecomposable objects given by 
\[  \ind \mathcal C_Q =\ind \mathbb CQ^{op} \,\cup\, \{\Sigma P_i, i\in Q_0\},   \] 
where $P_i$ is the indecomposable  projective $\mathbb CQ^{op}$-module associated with the vertex $i$. We will use $S_i$ to denote the simple module associated with the vertex $i$.

The {\em Caldero-Chapoton map}, denoted by $X_{?}$, maps each object of $\mathcal C_Q$ to a Laurent polynomial in $\mathbb Q[x_1^{\pm 1}, \dots, x_n^{\pm 1}]$ satisfying the following conditions:
\begin{enumerate}
\item[-]  $X_{\tau P_i} =x_i$ for $i\in Q_0$.
\item[-]  $X_?$ induces a bijection of the set of isomorphism classes of indecomposable rigid objects of  $\mathcal C_Q$  and the set of cluster variables of the cluster algebra $\mathcal A_Q$. 
\item[-] $X_?$ induces a bijection of the set of isomorphism classes of basic cluster-tilting objects of  $\mathcal C_Q$  and the set of clusters of the cluster algebra $\mathcal A_Q$, and it is compatible with mutation. 
\item[-] for each almost split triangle $L\rightarrow N\rightarrow M\rightarrow\Sigma L,$ then
\[  X_LX_M=X_N+1. \]
\item[-] for rigid  indecomposable objects $L,M$ such that $\dim\Ext^1(L,M)=1$ with non-split triangles 
\[  L\rightarrow B \rightarrow M  \rightarrow \Sigma L, \quad \text{and}\;\,   M\rightarrow B' \rightarrow L  \rightarrow \Sigma M, \]
then  $X_LX_M = X_B+X_{B'}$.  In this case, we denote by $\ET(L,M) = (B,B')$ the non-split triangles above and call it the {\em exchange triangles}.
\end{enumerate}

Suppose that $Q_0 = \{ 1,\dots, n\}$ such that each $k$ is a sink in $\mu_{k-1}\dots\mu_1(Q)$, and define $\mu_c := \mu_n\dots\mu_1$, which is called the {\em Coxeter mutation} in \cite{igusa2021frieze}.  For $t\in \mathbb N$, denote by $\mathbf x_t$ the cluster 
\[\mathbf x_t = (x_{1,t},\dots,x_{n,t}) := \mu^t_c(\mathbf x)\] 
obtained from $\mathbf x$ by applying $\mu_c$ for $t$ times. In the following notations, arrows in the subscripts are considered in $Q$.  By the mutation formula,  for 
any $t$ and $i$,  we have that 
\[   x_{i,t}x_{i,t+1} = \prod\limits_{j\rightarrow i}x_{j,t}\cdot   \prod\limits_{i\rightarrow k}x_{k,t+1}  +1.                \]

For $t\in \mathbb N$ and $i\in Q_0$, there is an almost split triangle
\[ \tau^{-t+1}P_i  \rightarrow  \bigoplus_{j\rightarrow i}\tau^{-t+1}P_j\oplus \bigoplus_{i\rightarrow k}\tau^{-t}P_k \rightarrow \tau^{-t}P_i \rightarrow    \Sigma \tau^{-t+1}P_i.\]

Then \[X_{\tau^{-t+1}P_i}X_{\tau^{-t}P_i} = \prod\limits_{j\rightarrow i}X_{\tau^{-t+1}P_j}\cdot \prod\limits_{i\rightarrow k}X_{\tau^{-t}P_k} +1. \]

Thus  $x_{i,t} = X_{\tau^{-t+1}P_i}$ for $t\in \mathbb N$ and $i\in Q_0$.

\subsection{Affine varieties and Gr\"obner bases}
In this paper, we only work over the field $\mathbb{C}$ of complex numbers  and the affine space $\mathbb C^n$. We will use  $\mathbb C^*:=\mathbb C\setminus\{0\}$ to denote the multiplicative group of $\mathbb C$. 

Let $\mathbb C^n$ be the affine $n$-space and $A=\mathbb{C}[x_1,\dots,x_n]$ the polynomial ring with $n$ variables. An {\em affine variety}  (or simply {\em variety}) is the common zero locus of a collection of polynomials in $A$. Beware that we do not require an affine variety to be irreducible. We denote by $K(X)$ the function field for any irreducible affine variety $X$.

 For a subset $T\subset A$, the common zero locus of polynomials in  $T$ is denoted by $\Z(T)$. For a subset  $V\subset \mathbb C^n$, the ideal of polynomials in $A$ vanishing on $V$ is denoted by $\I(V)$. 
It is well-known that the Zariski closure of  $V$ is  the set $\Z(\I(V))$.

\begin{definition}
	Let $X\subset \mathbb C^n$ be an affine variety, we say
	\begin{enumerate}
		\item[-]  $X$ is an {\em algebraic curve}, if $\dim X =1$.
		\item[-]  $X$ is {\em rational}, if $X$ is irreducible and  birationally  equivalent to $\mathbb{C}^r$ for some  $r$.
		\item[-]  $X$ is a {\em rational curve}, if $X$ is an algebraic curve and $X$ is rational, or equivalently $X$ is irreducible and  birationally  equivalent to $\mathbb{C}^1$.
\end{enumerate}                                                                                         
\end{definition}

\begin{lemma}\label{buqueding}
	Let $f: \mathbb  C^m \dashrightarrow \mathbb C^n$ be a rational map represented  by 
	$(U, f)$ where $U$ is an nonempty subset of $\mathbb C^m$ and $f: U\rightarrow \mathbb C^n$ is a morphism.  Then for any nonempty subsets $S_1, S_2, S\subset U$, we have that
	\begin{enumerate}
		\item\ $\dim \overline{S}\geq \dim \overline{f(S)}.$
		\item\ if $\overline{S_1} = \overline{S_2}$, then $\overline{f(S_1)}  =  \overline{f(S_2)}$.
	\end{enumerate}

	If, in addition, $f$ is birational with the inverse represented by $(V,g)$ and $f(S)\subset V$, we have that  $\dim \overline{S}= \dim \overline{f(S)}$. 
\end{lemma}
\begin{proof}
	(1) Assume that $X = \overline{S}$  has the irreducible decomposition 
	$X= X_1\cup \dots  \cup X_r$.  Since we have that 
	\[   \overline{f(S)} \subset \overline{f(X\cap U)}  =   \bigcup\limits_{i=1}^r\overline{f(X_i\cap U)},  \] 
	it is enough to prove that  $\dim X \geq  \dim \overline{f(X\cap U)}$ when $X$ is  an irreducible affine variety and  $X\cap U$ is not empty.  
	
	Let $Y =  \overline{f(X\cap U)}$.  Since $X\cap U$ is irreducible and $f$ is continuous,  it follows that $Y$ is also irreducible.
	We  have a dominant rational map $f:  X\dashrightarrow Y$ between irreducible affine varieties and thus an injective field homomorphism \[ f^*:  K(Y)   \hookrightarrow K(X). \]

	Thus we have 
	\[ \dim Y = \trdeg_{\mathbb C}K(Y) \leq \trdeg_{\mathbb C}K(X)  =\dim X,\]
	here $\trdeg_{\mathbb C}$ denotes the transcendence degree of the function field over $\mathbb C$. 
	
	(2) It is clear that $\I(S_1) = \I(S_2)$.
	Since for any polynomial $h\in A$, the map $h\circ f$ is regular on the open subset $U$, there are polynomials $H,G\in A$ such that 
	\[h(f(\mathbf x))  = \frac{H(\mathbf x)}{G(\mathbf x)}  \quad \text{on   $U$,}  \,\, \Z(G)\cap U=\emptyset.\]

	Thus we have that 
	\[h(\mathbf x)  \in   \I(f(S_2))\iff    H(\mathbf x) \in \I(S_2 )\iff H(\mathbf x) \in \I(S_1 )\iff  h(\mathbf x) \in \I(f(S_1)).  \]

	This completes the proof.
\end{proof}

\begin{definition}
	We say an affine variety  $X$  is {\em parametric}, if there are rational functions $r_1, \dots, r_n\in \mathbb{C}(x_1,\dots,x_m)$  with $r_i = \frac{f_i}{g_i}$, $f_i,g_i\in  \mathbb{C}[x_1,\dots,x_m]$ $(1\leq i\leq n)$  and $g=g_1\dots g_n\neq 0$   such that $X$ is the Zariski closure 	of  the subset  $Y\subset \mathbb C^n$  given by 
	\begin{equation}\label{densey} Y : =\{(r_1(\mathbf x), \dots, r_n(\mathbf x))\in \mathbb C^n \,|\, \mathbf x\in \mathbb C^m\setminus \Z(g)\}.\end{equation}

	Moreover, we say $X$ is {\em nontrivially parametrized by rational functions $r_1, \dots, r_n$},  if at least one of these rational  functions is  not a constant. 
\end{definition}

Not all of affine varieties can be parametrized by rational functions, and a parametric variety may have different parametrizations. The following Lemma shows only irreducible affine varieties can be parametrized and their dimensions are controlled by the number of parameters.

\begin{lemma}\label{xx}
	Suppose that $X\subset \mathbb C^n$  is an affine variety  parametrized by 
	rational functions  $r_1, \dots, r_n\in \mathbb C(x_1,\dots,x_m)$ and 
	$r_i =\frac{f_i}{g_i}$ with $f_i,g_i\in \mathbb C[x_1,\dots,x_m]$ and $g_i\neq 0$ for all $1\leq i\leq n$. Then 
	\begin{enumerate}
		\item\  $X$ is irreducible;
		\item\ $\I(Y)$ is prime, where $Y$ is given by (\ref{densey});
		\item\ $\dim X\leq m$.
	\end{enumerate}
\end{lemma}
\begin{proof}
	(1) $\&$ (2) :
	Let $\varphi: \mathbb C^m\dashrightarrow \mathbb C^n$ be the rational map given by:
	\[ \begin{split} \varphi:  \quad &U \longrightarrow \mathbb C^n \\  &\mathbf x \longmapsto  (r_1(\mathbf x),\dots,  r_n(\mathbf x)),   \end{split}\]
	where $U=\mathbb C^m \setminus \Z(g_1\dots g_n)$ is the nonempty open subset of $\mathbb C^m$.  Since $\varphi$ is a morphism from the irreducible variety $U$ to $\mathbb C^n$ and therefore  $\varphi$ is continuous, then  $X = \overline{Y}=\overline{\varphi(U)}$ is irreducible.     
	  Hence $\I(Y)  =  \I(\overline Y)$ is prime.
	
	(2) By Lemma \ref{buqueding}, we have that 
	\[ \dim X = \dim \overline{\varphi(U)} \leq  \dim \overline{U} = \dim \mathbb C^m = m.\]
	This finishes the proof.
\end{proof}

The following Lemma is about the genus of  rational  curves.

\begin{lemma}[\cite{fulton2008curve}]\label{rationalgenus}
	Let $X$ be an irreducible algebraic curve.  Then  $X$ is rational if and only if  the genus of $X$ is zero.
\end{lemma}

A parametric affine variety is not necessarily rational, however   L{\"u}roth's Theorem guarantees that a parametric algebraic curve is always rational.

\begin{lemma}\label{rationalcurve}
	Let $\varphi : U=\mathbb C^1\setminus \Z(g) \rightarrow \mathbb C^n$ be the map given by 
	\begin{equation}\label{rationalpara} t\longmapsto (\frac{f_1(t)}{g_1(t)},\dots, \frac{f_n(t)}{g_n(t)}) \end{equation}
	with $f_1(t), \dots, f_n(t), g_1(t),\dots, g_n(t) \in \mathbb{C}[t]$ such that the product $g(t)=g_1(t)\dots g_n(t)$ is nonzero and $f_i(t)$ and $g_i(t)$ have no common roots for each $i$. Suppose that not all of these rational functions  $\frac{f_i(t)}{g_i(t)}\, (1\leq i\leq n)$  are constants and $S\subset U$ is an arbitrary  infinite subset, then we have that
	\begin{enumerate}
		\item\ $\overline{\varphi(U)} = \overline{\varphi(S)}$;
		\item\ $X = \overline{\varphi(S)}$ is a rational curve.  In particular, the genus of $X$ is zero.
	\end{enumerate} 
\end{lemma}
\begin{proof}
	(1)  It is clear that  $\varphi(S)$ is an infinite set. Indeed, suppose that $r(t) = f(t)/g(t)$ is not a constant. If  $r(S)$ is a finite set, then we may assume that  $r(t) = c$ is a constant on an infinite subset $S'\subset S$.  Let $h(t) = f(t) -cg(t) \in \mathbb C[t]$.  It follows that $h(t)$ must be zero and this is a contradiction.		
Therefore $\dim \overline{\varphi(S)}>0$.	By Lemma \ref{xx}, we have that \[0<  \dim \overline{\varphi(S)} \leq  \dim \overline{\varphi(U)}  \leq 1.\] 

Thus  $\dim \overline{\varphi(S)} = \dim \overline{\varphi(U)}
	=1$. Since $\overline{\varphi(S)}$ is a closed subset of the irreducible affine  variety  $\overline{\varphi(U)}$, we have  $\overline{\varphi(S)}=  \overline{\varphi(U)}$.
	
	(2) By Lemma \ref{xx}, $X$ is an irreducible algebraic curve.   The dominant  rational map \[\varphi: \mathbb C^1\dashrightarrow  X =  \overline{\varphi(U)}\]
	between two irreducible affine varieties induces an  injective  homomorphism of function fields
	\[   \varphi^*: K(X)  \hookrightarrow  \mathbb C(x).\]

	The conclusion that  $\varphi^*$ is an isomorphism    follows from the L{\"u}roth's Theorem, which 
	claims a subfield of $\mathbb C(x)$ strictly containing $\mathbb C$ is $\mathbb C$-isomorphic to $\mathbb C(x)$.   Thus $X$ is a rational curve.
\end{proof}

The first conclusion in Lemma \ref{rationalcurve} can be false if there are $m$ parameters with $m\geq 2$. Indeed, if $S\subset \mathbb C^m$ is an infinite set, then $\varphi(S)\subset \mathbb C^n$ may be finite.

For a parametric curve as given in Lemma \ref{rationalcurve}, there are different ways in practice  to give their defining polynomials. The following lemma provides an  available method to obtain the prime ideal  $J$ such that $\Z(J) = \overline{\varphi(U)}$. 

\begin{lemma}[Elimination]\label{ratioanlimplicit}
	Let $\varphi$ be given as in (\ref{rationalpara}), and
	let $\widetilde{J}$ and $J$  the ideals given as follows:
	\[\begin{split}  \widetilde{J} &= \langle   x_ig_i(t) - f_i(t), 1\leq i\leq n \rangle \subset \mathbb{C}[t,x_1,\dots,x_n],\\ J &= \widetilde{J}\cap \mathbb{C}[x_1,\dots,x_n] \subset \mathbb{C}[x_1,\dots,x_n].\end{split}\] 
	Then   we have that 
	\begin{enumerate}
		\item\   $\widetilde{J}, J$ are prime ideals and $\I(\varphi(U)) = J$;
		\item\   $\overline{\varphi(U)}  =  \Z(J) $.
	\end{enumerate}
\end{lemma}

Lemma \ref{ratioanlimplicit} provides an algorithm to find the ideal for parametric varieties with one parameter. More precisely, one may compute the Gr{\"o}bner basis for $\widetilde{J}$  with respect to the lexicographic order $t>x_1>\dots>x_n$ and  elements of the Gr{\"o}bner basis not involving  $t$  define the variety $\Z(J)$ as we will explain later.  So let us  recall some basic facts on Gr\"obner bases, and we refer to \cite{adams1994groebner} for more details.

  Let $A=\mathbb C[x_1,\dots,x_n]$ be the polynomial ring and we write $x_1>\dots> x_n$  as the  {\em lexicographic order}  $>$  on the set  $\{\mathbf x^{\mathbf a}, \mathbf a\in \mathbb N^n\}$  of  monomials   of $A$, which is a total order such that  \[x_1^{a_1}\dots x_n^{a_n} >  x_1^{b_1}\dots x_n^{b_n}\]  if  either $a_i=b_i$ for all $i$, or the leftmost nonzero entry of  $(a_1-b_1, \dots,a_n-b_n)$ is positive.  
  
  The lexicographic order is a monomial order (or term order, we refer to \cite{adams1994groebner} for more details),  and it is the only monomial order we consider in this paper. A  {\em block order} with respect to a partition $[1,n] = I_1\sqcup I_2$  is a lexicographic order such that $x_i> x_j$ for all $i\in I_1$ and $j\in I_2$.
  
  For each nonzero polynomial $h\in A$, the {\em leading monomial}  $\Lm(h)$  and  {\em leading term}  $\Lt(h)$ of   $h  =\sum\limits_{\mathbf a\in \mathbf N^n} c_{\mathbf a}\mathbf x^{\mathbf a}$ are defined  as follows:
\[\Lm(h) := \max\{ \mathbf x^{\mathbf a},   c_{\mathbf a}\neq 0   \}, \;\;\text{and}\;\; \Lt(h) := c_{\Lm(h) }\mathbf x^{\Lm(h) },\]
 where the maximum is taken with respect to the given lexicographic order. We also define $\Lt(0) =\Lm(0) = 0$ and $\mathbf x^{\mathbf a}>0$ for all $\mathbf a\in (\mathbb Z_{\geq 0})^n$ and arbitrary monomial order $>$.
 
  The following lemma gives a division algorithm on $\mathbb C[x_1,\dots,x_n]$ with respect to a given monomial order, one may see \cite{adams1994groebner} for the proof. 

\begin{lemma}[Division Algorithm]\label{divisionalgo}
Let $(f_1,\dots,f_m)$ be an ordered $m$-tuple of nonzero polynomials in $A=\mathbb C[x_1,\dots,x_n]$ and  $f\in A$ be an arbitrary polynomial. Then there exist polynomials $q_1,\dots, q_n, r\in \mathbb C[x_1,\dots,x_n]$ such that 
\[f=q_1f_1+\dots+q_mf_m+r,\]
with   either $r=0$ or $r$ is a linear combination (with coefficients in $\mathbb C)$  of monomials, none of which is divisible by any of $\Lt(f_1),\dots,\Lt(f_m)$ and 
\[     \Lm(f) =  \max\{(\Lm(q_i)\cdot\Lm(f_i)), 1\leq i\leq m;\; \Lm(r)                   \}.                      \]
\end{lemma}

 Let $I\subset A$ be an ideal. A subset $G=\{g_1,\dots,g_m\} \subset I$ is said to be a {\em Gr\"obner basis} of $I$ with respect to the  order $>$ as above if  the following monomial ideals are equal:
 \[ \langle  \Lt(g_i), \,\,  1\leq i\leq n\rangle =  \langle \Lt(f), f\in I  \rangle. \]
 
 If $G$ is a Gr\"obner basis of $I$, then $I$ is generated by $g_1,\dots,g_m$, that is, $I =\langle  g_i, 1\leq i\leq m \rangle$. 
  
    Let $f,g\in A$ be two nonzero polynomials,  the {\em $S$-polynomial $S(f,g)$} is defined  to be 
 \[S(f,g) = \frac{\lcm(\Lm(f), \Lm(g))}{\Lt(f)}f  -\frac{\lcm(\Lm(f), \Lm(g))}{\Lt(g)}g,\] 
 where $\lcm$ is the least common multiple, $c_{\Lm(f)}$ and $c_{\Lm(g)}$ are the coefficients of leading terms in $f$ and $g$, respectively. 

 Buchberger gave a criterion to determine if a set of polynomials is a Gr\"obner basis by  using $S$-polynomials.
 
 \begin{proposition}[Buchberger's Criterion]
Let $I$ be an ideal generated by $g_1,\dots,g_m$, and let $S_{ij} := S(g_i,g_j)$ the $S$-polynomial. Then $\{g_1,\dots,g_m\}$ is a Gr\"obner basis for $I$ if and only if for all $i< j$, the remainder of the division of  $S_{ij}$ on  $g_1,\dots,g_n$ (with a given order)  is zero. 
\end{proposition}
 
Let  $\pi_k: \mathbb C^n \rightarrow  \mathbb C^{n-k}$ be the projection mapping $(a_1, \dots,a_n)$ to  $(a_{k+1},\dots,a_n)$.  The following lemma about elimination on Gr\"obner basis is well-known.

\begin{lemma}[Elimination on Gr\"obner basis]\label{eligro}  Let $I\subset  \mathbb C[x_1,\dots,x_n]$ be an ideal and $G=\{g_1,\dots,g_m\}\subset I$ be a Gr\"obner basis of $I$ with respect to the order $x_1>\dots>x_n$. Then  $G\cap \mathbb C[x_{k+1},\dots,x_n]$  is a  Gr\"obner basis
	of     $I\cap \mathbb C[x_{k+1},\dots,x_n]$.
	
	  In particular, let $S\subset \mathbb C^n$ be a subset,  and  $G$ is a Gr\"obner basis of $\I(S)$ with respect to the order $x_1>\dots>x_n$,   then $G\cap \mathbb C[x_{k+1},\dots,x_n]$ is  a Gr\"obner basis of  $\I(\pi_k(S))$.
\end{lemma}
\begin{proof}
  The first statement on elimination is well-known, see \cite{adams1994groebner}.  
  For the second statement, it follows from the first statement and the easy observation   \[\I(\pi_k(S)) = \I(S) \cap \mathbb C[x_{k+1},\dots,x_n].\]
  This completes the proof.
  \end{proof}

\subsection{Linear recurrences}
Let us  briefly recall some facts on linear recurrences. Let $K$ be a field, for example, $K= \mathbb Q,\; \mathbb C$ or $\mathbb Q(x_1,\dots,x_n)$.
A sequence $(a_j)_{j\in \mathbb{N}}$ of elements in $K$ is said to satisfy a {\em  linear recurrence}  if there is a positive integer  $k$ and elements  $c_0, \dots, c_{k}\in K$ such that $c_ka_{j+k} + c_{k-1}a_{j+k-1} + \dots + c_1a_{j+1} +c_0a_j  = 0$ holds for all $j\in \mathbb{N}$. The polynomial $p(x) = c_kx^k +  c_{k-1}x^{k-1} + \dots +c_1x+ c_0\in K[x]$ is called the {\em characteristic polynomial} of this sequence.  

The set $I$ of  characteristic polynomials of all linear recurrences  satisfied by  a  sequence  $(a_j)_{j\in \mathbb N}$ is an ideal, and  if $I$ is  non-zero,  it is generated by a non-trivial monic polynomial of minimal degree in $I$, which is called the {\em minimal characteristic polynomial}.  

In this paper, we will take $K=\mathbb C$ if not specified.

\begin{lemma}\label{lr}
	Let  $p(x) = x^k +  c_{k-1}x^{k-1} + \dots +c_1x+ c_0 \in \mathbb{C}[x]$ be a polynomial  with $c_0\neq 0$. Factor $p(x)$ over  $\mathbb{C}$ as 
	\begin{equation}\label{charpoly} p(x)  =    (x-r_1)^{m_1}(x-r_2)^{m_2} \dots (x-r_l)^{m_l}, \end{equation}
	where $r_1, r_2, \dots, r_l$ are distinct nonzero complex numbers, and $m_1, m_2, \dots, m_l$ are positive integers.  Then a sequence $(a_j)_{j\in \mathbb{N}}$  satisfies the linear recurrence with characteristic polynomial $p(x)$ if and only  if there are polynomials  $g_1(x), g_2(x), \dots, g_l(x)\in\mathbb{C}[x]$ with $\text{deg}(g_i)\leq m_i-1$ for $i = 1, 2, \dots, l$ such that for all $j\in \mathbb{N}$, one has  
	\[a_j  =   g_1(j) r_1^j + g_2(j)r_2^j+\dots+g_l(j)r_l^j .\]
	
	In particular, if $p(x)$ has no multiple roots, then $(a_j)_{j\in \mathbb{N}}$  satisfies the linear recurrence with characteristic polynomial $p(x)$ if and only  if there are  complex numbers $b_1,\dots,b_k$ (not depending on $j$) such that for all $j\in \mathbb{N}$, one has  
	\[a_j  =    b_1r_1^j +  b_2r_2^j+\dots  +b_kr_k^j .\]
\end{lemma}

Let us concentrate on the following special case of linear recurrences.

\begin{lemma}\label{lr3}
	Suppose that $(a_j)_{j\in \mathbb{N}}$ satisfies the linear recurrence with characteristic polynomial $p(x) = x^2  - cx +1$. We have that  
	\begin{enumerate}
		\item\ if $c\notin \{2,-2\}$,  there exists a Laurent polynomial of the form  $f(x) = b_1x^{-1}+ b_2x$ and a complex number $\rho\notin \{1,-1,0\}$  such that  $a_j = f(\rho^j)$   for all $j\in \mathbb{N}$;
		\item\ if $c=2$, there is a linear polynomial $f(x) = b_0+ b_1x$ such that  $a_j = f(j)$ for all $j\in \mathbb{N}$;
		\item if $c=-2$, there is a linear  polynomial $f(x) = b_0+b_1x$ such that  $a_j = f(j)$ for all even numbers $j\in \mathbb{N}$,  and  $a_j = -f(j)$   for all odd numbers $j\in \mathbb{N}$.
	\end{enumerate}
\end{lemma}
\begin{proof}
	If $c\notin \{2, -2\}$, then $p(x) = (x-\rho)(x-\frac{1}{\rho})$ with some  complex number $\rho \notin \{0, 1, -1\}$.  By Lemma \ref{lr}, there are constants $b_1, b_2$ such that $a_j  =    b_1{\rho}^{-j} +  b_2{\rho}^j$ holds for all $j\in \mathbb{N}$. Let  $f(x) = b_1x^{-1} + b_2 x$,  then $a_j =f(\rho^j)$ for  all $j\in \mathbb{N}$.
	
	If $c=2$, then $p(x) = (x-1)^2$.  By Lemma \ref{lr},  there  is a polynomial $f(x)$ with $\text{deg}(f) \leq 1$ such that  for all $j\in \mathbb{N}$, $a_j = f(j)$.
	
	If $c=-2$, then $p(x) = (x+1)^2$.    By Lemma \ref{lr},  there is a  a polynomial $f(x)$ with $\text{deg}(f) \leq 1$ such that  for all $j\in \mathbb{N}$, $a_j = f(j)(-1)^j$.
\end{proof}

Let $(a_n)_{n\in \mathbb Z}$ be a sequence of complex numbers. Suppose that there is an positive integer $m$ and some  complex numbers $c_0,\dots, c_{k-1}$ with $c_0\neq 0$  such  that 
\[a_{n+km} + c_{k-1}a_{n+(k-1)m} + \dots + c_1a_{n+m} +c_0a_n  = 0\]
holds for all  $n\in \mathbb Z$.
Then the sequences $(a_{t+jm})_{j\in \mathbb N}$ ($t\in\mathbb Z$) satisfy the 
linear recurrence and share a common characteristic polynomial given by 
\[p(x) = x^k +  c_{k-1}x^{k-1} + \dots +c_1x+ c_0 = (x-r_1)^{m_1}(x-r_2)^{m_2} \dots (x-r_l)^{m_l} \]
for all  $t\in \mathbb Z$. By Lemma \ref{lr}, there are polynomials 
$g_{i,t}(x)$ ($1\leq i\leq m$) such that 
\[a_{t+jm} =  g_{1,t}(j) r_1^j + g_{2,t}(j)r_2^j+\dots+g_{l,t}(j)r_l^j  \]
for all $j\in \mathbb{N}$. Here  the roots $r_1,\dots,r_l$ are independent of the choice of $t\in \mathbb Z$. 

If in addition, assume that  $p(x)$ is of the form $p(x)= x^{2m}-cx^m+1$. 
That is, the sequence $(a_n)_{n\in \mathbb Z}$ satisfies that
\[   a_{n+2m} -ca_{n+m} +a_n =0  \]
holds for all  $n\in \mathbb Z$.  For each $t\in \mathbb Z$, the sub-sequence $(a_{t+jm})_{j\in \mathbb Z}$ satisfies a linear recurrence with characteristic polynomial  given by 
\[f(x) = x^2- cx+1.\] 
If $c\neq \pm2$, there is a complex number $\rho \notin \{1,-1,0\}$
such that for each $t\in \mathbb{Z}$, there is a Laurent polynomial $f_t(x)$ such that  $a_{t+jm} = f_t(\rho^j)$ for all $j\in \mathbb N$.  If $c= 2$,  for each $t\in \mathbb{Z}$, there is a  polynomial $f_t(x)$ such that  $a_{t+jm} = f_t(j)$ for all $j\in \mathbb N$.   If $c= -2$,  for each $t\in \mathbb{Z}$, there is a  polynomial $f_t(x)$ such that  $a_{t+jm} = f_t(j)$ for all even $j\in \mathbb N$ and $a_{t+jm} = -f_t(j)$ for all odd $j\in \mathbb N$.  

\section{Generalized frieze varieties}\label{fv}
In this section, we define and study generalized frieze varieties for  quivers associated to a cluster automorphism, and we explain that it extends the definition of (generalized) frieze variety defined in  \cite{lee2020frieze, igusa2021frieze}.

\subsection{Generalized frieze varieties associated to cluster automorphisms}
Let $\mathcal{A} =\mathcal A(\mathbf x, Q)$ be a cluster algebra, 
and  $f: \mathcal{A} \rightarrow \mathcal{A}$   a cluster automorphism. For any $t\in \mathbb{N}$,  the cluster $\mathbf x_t  = (x_{1,t}, \dots, x_{n,t})$ is given by $f^t(\mathbf x)$. Notice that the cluster variables $x_{i,t}$ are Laurent polynomials of $\mathbf x$ with positive coefficients.  For $\mathbf a\in (\mathbb C^*)^n$, we will write $\mathbf x_t(\mathbf a) := \mathbf x_t|_{\mathbf x = \mathbf a} \in \mathbb C^n$, and we also use the notation $f^t(\mathbf a)$ to denote  $\mathbf x_t(\mathbf a)$.

\begin{definition}\label{defgfv}
	Let  $\mathcal A$, $f$ and $\mathbf x_t$ be given as above.
	\begin{enumerate}
		\item  A point $\mathbf a \in (\mathbb C^*)^n$ is called a {\em general specialization}\footnote{Igusa and Schiffler call a general specialization a generic specialization in \cite{igusa2021frieze}.  We would like to distinguish it from generic points in topology,  since a generic point of a variety  is usually defined to be a point which is dense in that variety.} of $(\mathbf x, Q)$ with respect to $f$ (or simply $(Q,f)$), if  $\mathbf x_t(\mathbf a) \in (\mathbb C^*)^n$ for all $t\geq 0$.
		\item For  $\mathbf a\in (\mathbb C^*)^n$, we define the {\em generalized frieze variety}  $X(Q, f, \mathbf a)$  of $(\mathbf x, Q)$ (or simply $\mathbf x$ or $Q$)  with respect to $f$ and $\mathbf a$ to be the Zariski closure of  the set  \[\{\mathbf x_t(\mathbf a), t\geq 0\}\subset \mathbb C^n.\]
	\end{enumerate}
\end{definition}

Even though the generalized frieze variety can be defined  for every $\mathbf a\in  (\mathbb C^*)^n$,  we mainly consider the generalized frieze variety for a general specialization in this paper.  

By Theorem \ref{laurentpositivity}, it is clear that  the set of general points  is a nonempty  set  containing $\mathbb R_{>0}^n$. In particular,
the vector $\mathbf 1 = (1,1,\dots, 1)$ is a general specialization for every cluster and every cluster automorphism. In this case, we simply use $X(Q,f)$  to denote $ X(Q, f, \mathbf 1)$ and call it the {\em frieze variety with respect to $f$}.  

We also notice that the cluster algebras $\mathcal A(\mathbf x, Q)$ and $\mathcal A(\mathbf x, Q^{\op})$ are equal and share the same clusters, so we usually do not distinguish $X(Q,f,\mathbf a)$ and $X(Q^{\op},f,\mathbf a)$.

\begin{lemma}[\cite{igusa2021frieze}]\label{atob}
	Let $\mathbf a, \mathbf b\in  (\mathbb C^*)^n$ be two general specializations of $Q$ with respect to  the cluster automorphism $f$.  If $\mathbf a\in X(Q,f,\mathbf b)$, then we have that 
	\[
	X(Q,f,\mathbf a)\subset X(Q,f,\mathbf b).
	\]
\end{lemma}
\begin{proof}
	For $j\in \mathbb N$, let
\[\mathbf a_j = f^j(\mathbf x)|_{\mathbf x=\mathbf a}\in (\mathbf C^*)^n.\]
It is clear that $\mathbf a_j$'s are general specializations for $Q$ with respect to $f$, thus it is enough to prove that $\mathbf a_1\in X(Q,f,\mathbf b)$. The proof is similar to that given in [Lemma 3.1, \cite{igusa2021frieze}]. 
\end{proof}

The following result proved by Igusa and Schiffler in \cite{igusa2021frieze} shows that irreducible components of generalized frieze varieties are permuted by the cluster automorphism.
\begin{lemma}[\cite{igusa2021frieze}]\label{isinvariant}
	Let $\mathbf a$ be a general specialization for $Q$ with respect to $f$ and $X=X(Q,f,\mathbf a)$ the generalized frieze variety.  Then there is an integer $m\geq 1$
such that  $X=X_0\cup \dots \cup X_{m-1}$, where $X_i = X(Q, f^m,f^i(\mathbf a))$ for $i\in [0,m-1]$ satisfies the following conditions:
\begin{enumerate}
	\item[(i)] $X_i$ is irreducible,  $0\leq i\leq m-1$;
	\item[(ii)] $X_i\nsubseteq X_j$ for $i\neq j$;
	\item[(iii)] for each $j\in \mathbb N$, $f^j(\mathbf a)$ only  lies in one $X_i$ such that $j\equiv i$ modulo $								m$;
   \item[(iv)] $\dim X = \dim X_0= \dots=\dim X_{m-1}$.
	\end{enumerate}
	\end{lemma}

We have the following results.

\begin{proposition}\label{inverse}
	Let $\mathbf a$ be a general specialization for  $(Q, f)$. Then  for any $j\in \mathbb N$, 
	\[       X(Q, f, \mathbf a) =  X(Q,f, f^j(\mathbf a)).   \]
	If moreover $\mathbf a$ is also a general specialization for  $(Q, f^{-1})$, then 
	\[       X(Q, f, \mathbf a) =  X(Q,f^{-1}, \mathbf a).   \]
 \end{proposition}
\begin{proof}
We only need to prove the first statement  for $j=1$.
By Lemma \ref{isinvariant}, there is  a positive  integer $m$ such that 
$X(Q, f^m, \mathbf a)$ is irreducible. Since  $X(Q, f^m,f^m(\mathbf a))$ is a closed  subset of  $X(Q, f^m, \mathbf a)$   and clearly  $\dim X(Q, f^m,f^m(\mathbf a))  =  \dim X(Q, f^m, \mathbf a)$, thus they are equal.  Hence 
\[     \mathbf a \in   X(Q, f^m, \mathbf a) =    X(Q, f^m,f^m(\mathbf a)) \subset    X(Q,f, f(\mathbf a)).      \]
So it is clear that $X(Q,f,\mathbf a) = X(Q,f,f(\mathbf a))$.

For the second statement,  it is enough to show that $f^{-j}(\mathbf a) \in X(Q,f,\mathbf a)$ for any $j\geq 0$. Indeed, by  the first statement we just proved, we have 
\[    X(Q, f, f^{-j}(\mathbf a))  =     X(Q, f, f^{-j+1}(\mathbf a)) =\dots = X(Q, f, \mathbf a).         \]

It is clear that  $f^{-j}(\mathbf a) \in X(Q,f,\mathbf a)$ and thus  $ X(Q,f^{-1}, \mathbf a)\subset  X(Q,f, \mathbf a)$.  The converse is similar. 
\end{proof}

By Proposition \ref{inverse},  if $\mathbf a$  is a general specialization
for both $f$ and $f^{-1}$, for example $\mathbf a\in (\mathbb R_{>0})^n$,   it is safe to  define the generalized frieze variety to be the Zariski closure of the set $\{f^t(\mathbf x)|_{\mathbf x=\mathbf a}, t\in \mathbb Z\}$, and it is the same to consider the Zariski closure of  $\{f^t(\mathbf x)|_{\mathbf x=\mathbf a}, t\geq j\}$ for some $j\in \mathbb Z$.
 
 \begin{corollary}
 	A generalized frieze variety $X(Q,f,\mathbf a)$
is irreducible if and only if  $f^j(\mathbf a)$ and $f^{j+1}(\mathbf a)$ belong to the same irreducible component of  	$X(Q,f,\mathbf a)$ for some $j\in \mathbb N$.
 	\end{corollary}
 \begin{proof}
 Let $\mathbf b= f^j(\mathbf a)$ such that $\mathbf b$ and $f(\mathbf b)$ belong to the same irreducible component of $X(Q,f,\mathbf a)=X(Q,f,\mathbf b)$. By Lemma  \ref{isinvariant},  $X(Q,f,\mathbf a)$ is irreducible.
 \end{proof}
\begin{lemma}\label{ir} 
 If the generalized frieze variety $X(Q,f,\mathbf a)$ is irreducible, then  for any $r\geq 1$, we have that \[     X(Q,f,\mathbf a) =X(Q,f^r,\mathbf a)    .\]
\end{lemma}
\begin{proof}
Let $Y_i =  X(Q, f^r,  f^i(\mathbf a))$ for $0\leq i\leq r-1$, then 
$X(Q,f,\mathbf a)  =  Y_0 \cup \dots \cup Y_{r-1}$.  By Lemma \ref{buqueding}, $\dim Y_0 = \dots =\dim Y_{r-1}$, and hence $\dim X(Q,f,\mathbf a) = \dim Y_0$. Since $X(Q,f,\mathbf a)$ is irreducible, thus $X(Q,f,\mathbf a) = Y_0$.
\end{proof}

More generally, we have the following result. 
\begin{proposition}
With the notations as in Lemma \ref{isinvariant}.
Suppose that $X=X(Q,f, \mathbf a)$ has the irreducible decomposition  $X=X_0\cup \dots \cup X_{m-1}$. For any $k\geq 1$, the generalized frieze variety
 $X(Q,f^k,\mathbf a)$ has the irreducible decomposition
\[  X(Q,f^k,\mathbf a)  =\bigcup\limits_{0\leq j< s} X_{jd},           \]
where  $d= \gcd(m,k)$ and $m = d\cdot s$. 
 		\end{proposition}
\begin{proof}
Let $\lcm(m,k) =m r =ks$. We have that  
\[\begin{split} X(Q,f^k,\mathbf a)  &= \bigcup\limits_{1\leq i\leq s} X(Q,f^{mr}, f^{ik}(\mathbf a))\\
	&= \bigcup\limits_{1\leq i\leq s} X(Q,f^{m}, f^{ik}(\mathbf a)).
\end{split}	\]

The second equality follows from Lemma \ref{ir} and the fact that $X(Q,f^m, f^{ik}(\mathbf a)) = X_j$ for some $j$ by Lemma \ref{inverse} which is irreducible.  Since $\{ik, 1\leq i\leq s\}$  modulo $m$ is the same as $\{jd, 0\leq j<s\}$ modulo $m$,  the statement holds by Lemma \ref{inverse}.
\end{proof}

The following results will  show that numbers of irreducible components and dimensions of  generalized  frieze varieties  are preserved  under mutation.

\begin{lemma}\label{unmutations}
Assume that  the generalized frieze variety $X(Q,f,\mathbf a)$ is irreducible.  Let   $(\mathbf x', Q') = \mu_{i_k}\dots\mu_{i_1}(\mathbf x, Q)$ such that $\mathbf b:= \mathbf x'|_{\mathbf x=\mathbf a}$ is a general point for $(\mathbf x',Q')$ with respect to $f$. Then $X(Q',f,\mathbf b)$ is irreducible and birationally equivalent to $X(Q,f,\mathbf a)$.
\end{lemma}
\begin{proof}  Let  $X= X(Q,f,\mathbf a)$ and $X' = X(Q',f,\mathbf b)$ for simplicity.
	We may view mutations as birational maps on affine spaces, and let $g: \mathbb C^n\dasharrow \mathbb C^n$ and $h: \mathbb C^n\dasharrow \mathbb C^n$ denote the birational maps defined on the torus $(\mathbb C^*)^n$ which are induced by  $\mu_{i_k}\dots\mu_{i_1}$ and $\mu_{i_1}\dots\mu_{i_k}$, respectively. 
	Since $f$ is a cluster automorphism, then  $f^t(\mathbf b) =  g(f^t(\mathbf a))$ for $t\in \mathbb N$.  Denote by $S$ the set $\{f^t(\mathbf a), t\in \mathbb N\}$, then it is obvious that $S=h(g(S)) \subset Y$ and $g(S) = \{f^t(\mathbf b), t\in \mathbb N\}$. 

	 Let $Y= X\cap (\mathbb C^*)^n$, then $\overline{g(Y)}$ is irreducible by the irreducibility of $Y$ and continuity of $g$,  and by Lemma \ref{buqueding}, we have that
	\[  \dim   X =\dim \overline{h(g(S))} \leq  \dim \overline{g(S)}\leq   \dim  \overline{g(Y)} \leq \dim \overline{Y} = \dim X.                \]
	
	Hence $\dim X' = \dim \overline{g(Y)}$ and  we therefore have $X' = \overline{g(Y)}$ is irreducible. The birational equivalence of $X$ and $X'$ is clear.
\end{proof}

By Lemma \ref{isinvariant} and Lemma \ref{unmutations}, we obtain the following result.

\begin{proposition}
Let $\mathbf a$ be a general point for $(\mathbf x, Q)$ with respect to $f$, and let $\mathbf b :=  \mathbf x'(\mathbf a)$ be a  general point for   $(\mathbf x', Q')$ with respect to $f$, where   $(\mathbf x', Q') = \mu_{i_k}\dots\mu_{i_1}(\mathbf x, Q)$.  Then  $X(Q,f, \mathbf a) = \cup_{i=0}^{m-1} X(Q,f^m,f^i(\mathbf a))$ is the   irreducible decomposition of  $X(Q,f, \mathbf a)$ if and only if  $X(Q',f, \mathbf b) = \cup_{i=0}^{m-1} X(Q',f^m,f^i(\mathbf b))$ is the  irreducible decomposition of  $X(Q',f, \mathbf b)$, and  $X(Q,f^m,f^i(\mathbf a))$ is birationally equivalent to  $X(Q',f^m,f^i(\mathbf b))$ for each $i$.
\end{proposition}

\begin{example}\label{markovexa}
Let $Q$ be the Markov quiver given by 
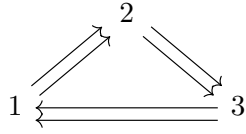
\[
\begin{tikzcd}& 2\arrow[dr,shift right=0.5ex] \arrow[dr,shift left=0.5ex]&\\
	1\arrow[ur,shift left=0.5ex]  \arrow[ur,shift right=0.5ex]  & & 3\arrow[ll,shift left=1.0ex]\arrow[ll]  \end{tikzcd}
\]

In this case, each mutation sequence induces a cluster  automorphism.

 In particular, $\mu_1$ induces a cluster automorphism  $f_1$ given by
\[ f_1(x_1) = \frac{x_2^2+x_3^2}{x_1},\quad f_1(x_2) = x_2,\quad f_1(x_3)=x_3. \]

For any general point $\mathbf a=(a,b,c)\in (\mathbb C^*)^3$  of $(Q,f_1)$, the generalized frieze variety $X(Q,f_1,\mathbf a)$ consists of  only one point $\mathbf a$ if and only if $a^2=b^2+c^2$, for example $\mathbf a=(5,4,3)$. Otherwise, $X(Q,f_1,\mathbf a)$  consists of two points  $\mathbf a$ and $f(\mathbf a)$. 

Suppose that $f$ is induced by $\mu_2\mu_1$,  that is $f(\mathbf x) = \mu_2\mu_1(\mathbf x)$.
This is the same as the Kronecker quiver with a frozen vertex. Let $\mathbf a=(a,b,c)\in (\mathbb C^*)^3$ be a general point, then $X(Q,f,\mathbf a)$ is contained in the zero locus of $h_1$  and $h_2$ which are given by 
\[    h_1= x_1^2+x_2^2-\frac{a^2+b^2+c^2}{ab} x_1x_2+c^2,\quad  h_2=x_3-c.              \] 

Let 
\[     d: =(\frac{a^2+b^2+c^2}{ab})^2 -2.            \]

The polynomial $h_1$ is irreducible if and only if $d=2$.  

The generalized frieze variety $X=X(Q,f,\mathbf a)$
has dimension $0$ if and only if $d$ satisfies that $d\neq  2$ and the roots of $x^2-dx+1$  are  roots of unity, that is  $d= 2\cos\frac{2k\pi}{m}$ for some positive integers $k,m$ with $1\leq k\leq m-1$.  Otherwise, $\dim X=1$ and  $X$ is isomorphic to a line if  $d=2$ and is isomorphic to a conic if $d\neq 2$.  More concretely, we  illustrate some examples with $\dim X=0$  and all cases with $\dim X=1$ as follows. 

Let $d = 0$, then  $f^2(\mathbf a)= (-a, -b,c)$ and
\[  X(Q,f,\mathbf a) =\{ \mathbf a, \;  f(\mathbf a), \; f^2(\mathbf a),\;  f^3(\mathbf a)\}.\]

Let $d = \pm 1$, then 
\[  X(Q,f,\mathbf a) =\{ \mathbf a, \; f(\mathbf a),\; f^2(\mathbf a)  \}.\]

Let $d=-2$, this is equivalent to $a^2+b^2+c^2=0$, we have that
\[X(Q,f,\mathbf a) =\{\mathbf a\} \cup \{(-a,-b,c)\}.
\]

Let $d=2$, this is equivalent to either $(a+b)^2+c^2=0$ or $(a-b)^2+c^2=0$. In this case,  $X(Q,f,\mathbf a)$ is a smooth irreducible curve  of degree $1$ given as follows, which is isomorphic to a  line in the plane $\mathbb C^2$.
\[X(Q,f,\mathbf a) = \begin{cases} \Z(h_2,\;\, x_1+x_2-a-b), \quad &\text{if  $(a+b)^2+c^2=0$},\\
	\Z(h_2, \;\,x_1-x_2-a+b), \quad &\text{if  $(a-b)^2+c^2=0$}. 
	\end{cases}\]

Suppose $d\in \mathbb C$ such that the roots of  $x^2-dx+1$ are not roots of unity.  In this case  $h_1$ is irreducible, and $X(Q,f,\mathbf a) = \Z(h_1,h_2)$  is a smooth irreducible curve of degree $2$, which is isomorphic to the conic $\Z(h_1(x,y)) \subset \mathbb C^2$.

For example, if  $\mathbf a=(\sqrt{2},3,\sqrt{5}\mathrm{i})$, then $d=0$ and
\[ X(Q,f,\mathbf a) = \{\mathbf a,  (2\sqrt{2},1,\sqrt{5}\mathrm{i}), ( -\sqrt{2},-3,\sqrt{5}\mathrm{i}), (-2\sqrt{2},-1,\sqrt{5}\mathrm{i})\}.\]
If $\mathbf a=(\frac{\sqrt{2}}{2}\mathrm{i}, \frac{\sqrt{2}}{2}\mathrm{i},1)$, then $d=-2$ and $X(Q,f,\mathbf a) =  \{\mathbf a\}\cup \{   (-\frac{\sqrt{2}}{2}\mathrm{i}, -\frac{\sqrt{2}}{2}\mathrm{i},1)  \}$.  
If $\mathbf a=(3,4,\mathrm{i})$, then $d=2$ and $f^j(\mathbf a) =(2j+3,2j+4,\mathrm{i})$. Hence  $X(Q,f,\mathbf a) = \Z(x_3-\mathrm{i}, x_1-x_2+1)$.
While  in most cases, for example $\mathbf a\in (\mathbb R_{>0})^n$ (clearly $d>2$, the roots of $x^2-dx+1$ are not roots of unity),  $X(Q,f,\mathbf a)$ is the zero locus of $h_1$ and $h_2$.
\end{example} 

The quiver and its maximal green sequence in the following example was studied by Zhou in \cite{zhou2019scattering}.
\begin{example}\label{markovunfolding}
Let $Q$ denote the following quiver:
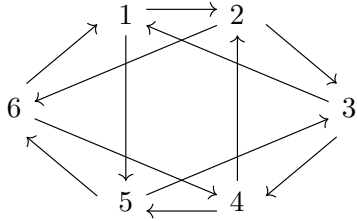
\[\begin{tikzcd}
	&1\arrow[r,shift left=0.6ex]\arrow[dd] &2\arrow[dr,shift left=1ex]\arrow[dll]&\\
	6\arrow[ur,shift left=1ex]\arrow[drr] &&&3\arrow[ull]\arrow[dl,shift left=1ex]\\
	&5\arrow[ul,shift left=1ex]\arrow[urr]&4\arrow[l,shift left=0.6ex]\arrow[uu]&
\end{tikzcd}\]

The  cluster DT automorphism $f$ induced by the maximal green sequence  \[(1, 3, 2, 4, 6, 5, 1, 6, 4, 3, 2, 5)\]  is given by 
\[\begin{split}
	f(x_1) = \frac{(x_1x_4+x_2x_5+x_3x_6)^2}{x_2x_3x_4x_5x_6},\quad f(x_2) = \frac{(x_1x_4+x_2x_5+x_3x_6)^2}{x_1x_3x_4x_5x_6}, \\ 
	f(x_3) = \frac{(x_1x_4+x_2x_5+x_3x_6)^2}{x_1x_2x_4x_5x_6},\quad
f(x_4)=	\frac{(x_1x_4+x_2x_5+x_3x_6)^2}{x_1x_2x_3x_5x_6},\\ 
f(x_5) = \frac{(x_1x_4+x_2x_5+x_3x_6)^2}{x_1x_2x_3x_4x_6},\quad f(x_6) = \frac{(x_1x_4+x_2x_5+x_3x_6)^2}{x_1x_2x_3x_4x_5}.
	\end{split}\]

It is clear that $f^2=id$. The generalized frieze variety  $X(Q,f,\mathbf a)$ consists of two points $\mathbf a$
and $f(\mathbf a)$. For example, we have that
\[X(Q,f,\mathbf 1) :=\{(1,1,1,1,1,1),\; (9,9,9,9,9,9)\}.\]
\end{example}

\subsection{Coxeter automorphisms}
Let $Q$ be an acyclic quiver. We may, without loss of generality,   assume that  the vertices are labeled by $1, 2, \dots, n$ such that if there is an arrow $i\rightarrow j$, then  $i>j$. The   Coxeter mutation $\mu_c : = \mu_n  \dots \mu_2 \mu_1$ is the composition of the mutation sequence  $\mu_1, \dots, \mu_n$.  For $t\in \mathbb{N}$, $\mu_c^t$ is defined inductively by  $\mu_c^t = \mu^{t-1}_c\cdot \mu_c$.
\begin{lemma}\label{coxeterautomorphism}
	Let $Q$ be an acyclic quiver, $\mathcal A=\mathcal A(\mathbf x, Q)$ the cluster algebra, and $\mu_c$ the Coxeter mutation.  Then there exists a unique  cluster automorphism
	$f_c: \mathcal A \rightarrow \mathcal A$ such that for all $t\geq 0$, one have that 
	\[f_c^t(\mathbf x)  =  \mu_c^t(\mathbf x). \]	
\end{lemma}
\begin{proof}
	The uniqueness is obvious. Let us prove the existence.
	
	Let $f_c : \mathbb{Q}(x_1,\dots, x_n) \rightarrow \mathbb{Q}(x_1,\dots, x_n)$ be the automorphism of the field $\mathbb{Q}(x_1,\dots, x_n)$ by sending $\mathbf x$ to $\mathbf x' = \mu_c(\mathbf x)$. Notice that $\mu_c(Q) = Q$, then it is easy to see that     \[f(\mu_{x}(\mathbf x))  =  \mu_{f(x)}(\mathbf x')\] for every $x\in \mathbf x$.
	By Lemma \ref{cllp}, we know that  $f_c$ restricts to a cluster automorphism of $\mathcal A$.   
	
	Since $f_c(\mathbf x) = \mu_c(\mathbf x)$, we  prove $f_c^t(\mathbf x)  =  \mu_c^t(\mathbf x)$ by induction on $t$. For $t\geq 1$, we have that 
	\[f_c^t(\mathbf x)= f_c(f_c^{t-1}(\mathbf x)) = f_c(\mu_c^{t-1}(\mathbf x)) =\mu_c^{t-1}(f_c(\mathbf x)) = \mu_c^t(\mathbf  x). \]
	
	This finishes the proof.
\end{proof}

Let $(\mathbf x, Q)$  be an acyclic seed and $f_c$ the induced cluster automorphism of $\ca_Q$ as above. Suppose that $i_1,\dots, i_n$  is an arbitrary admissible sequence of sinks, that is, $i_k$ is a sink of $\mu_{i_{k-1}}\dots \mu_{i_1}(Q)$ for any $k$, then it is easy to see we always have that  $f_c(\mathbf x) = \mu_{i_n}\dots\mu_{i_1}(\mathbf x)$.

\begin{lemma}
	Let $(\mathbf x, Q)$ be an acyclic seed and $f_c$ the Coxeter automorphism of $\ca=\ca (\mathbf x,Q)$. Suppose that $(\mathbf x', Q')$ is another acyclic seed of $\ca$ and $g_c$ is the Coxeter automorphism induced by $(\mathbf x',Q')$.  Then $f_c = g_c$.
\end{lemma}
\begin{proof}
	By [Corollary 4, \cite{caldero2006triangulatedii}],  acyclic seeds of $\ca$ form a connected subgraph of the exchange graph of $\ca$ and two acyclic seeds can be obtained from each other by a sequence of mutations at sinks or sources.  Thus $(\mathbf x, Q)$ and $(\mathbf x',Q')$ are related by a sequence of  mutations at sinks or sources. We may assume that they are related by a single mutation $\mu_k$, and we can take an admissible sequence  $i_1,\dots,i_n$ of sinks such that $i_1=k$ if $k$ is sink of $Q$ and  $i_n=k$ if $k$ is a source of $Q$.  
	
	If $k$ is a source of $Q$, then $k=i_n, i_1,\dots, i_{n-1}$ is an admissible sequence of sinks of $Q'$ and we therefore have that  
	 \[ f_c(\mathbf x') =f_c(\mu_k(\mathbf x)) = \mu_k\mu_ {i_n}\dots\mu_{i_1}\mu_k(\mathbf x') = \mu_ {i_{n-1}}\dots\mu_{i_1}\mu_k(\mathbf x') =g_c(\mathbf x').\]  
	 
	 If   $k$ is a sink of $Q$,  then $i_2,\dots, i_n, i_1=k$ is an admissible sequence of sinks of $Q'$ and we therefore have that  
	 \[ f_c(\mathbf x') =f_c(\mu_k(\mathbf x)) = \mu_k\mu_ {i_n}\dots\mu_{i_1}\mu_k(\mathbf x') =\mu_{k}\mu_ {i_n}\dots\mu_{i_2}(\mathbf x') = g_c(\mathbf x').\]  
	 
	 This completes the proof.
\end{proof}

We call the cluster  automorphism $f_c$ associated to the Coxeter mutation $\mu_c$  in  Lemma \ref{coxeterautomorphism} the {\em Coxeter automorphism}.  In the  following sections, we will mainly consider  generalized frieze varieties of affine quivers with respect to  Coxeter automorphisms.

By Lemma  \ref{coxeterautomorphism}, if $Q$ is an acyclic quiver and $\mathbf a$ a general specialization for $Q$ and $f_c$, then the  affine variety $X(Q, f_c, \mathbf a)$ is precisely the generalized frieze variety  $X(Q, \mathbf a)$ defined by Igusa and Schiffler in \cite{igusa2021frieze}. 
If in addition $\mathbf a = \mathbf 1$,
the variety $X(Q,f_c,\mathbf 1)$ is  precisely the frieze variety  $X(Q)$ defined for every acyclic quiver by Lee, Li, Mills, Schiffler and Seceleanu in \cite{lee2020frieze}.

The following Lemma is inspired by [Lemma 3.10,  \cite{assem2012automorphism}].

\begin{lemma}\label{mm}
	Let $Q$ be an acyclic quiver whose vertices are labeled by an admissible sequence of sinks $1,\dots,n$, and let  $\sigma \in S_n$ satisfy the following conditions:
	\begin{enumerate}
		\item[-] $\sigma^2=id$,
		\item[-] $\sigma$ induces an anti-automorphism of  $Q$,
		\item[-]  $\sigma(1),\dots,\sigma(n)$ is an admissible sequence of sources.  
	\end{enumerate}   
	Let  $f_{\sigma}$ be the inverse cluster automorphism induced by $\sigma$ and $f_c$  the Coxeter automorphism.  Then for any $t\in \mathbb Z$, we have that 
	\[    f^2_{\sigma} = id,\;\;\; \text{and}\;\;\;  f_{\sigma} f_c^tf_{\sigma} = f^{-t}.                             \]
\end{lemma}
\begin{proof}
	Since  $f_{\sigma}(x_i) = x_{\sigma(i)}$, it follows from $\sigma^2=id$ that $f_{\sigma}^2 = id$. It is enough to show that $f_{\sigma} f_cf_{\sigma} = f^{-1}_c$.  Indeed, we have that
	\[  f_{\sigma} f_cf_{\sigma}  (x_i)  =      f_{\sigma} (f_c  (x_{\sigma(i)}))  =   f_{\sigma}(\mu_n\dots\mu_1(x_{\sigma(i)}))=
	\mu_{\sigma(n)}\dots\mu_{\sigma(1)} (f_{\sigma}(x_{\sigma(i)}))   =  f_c^{-1}(x_i)                 \]
	holds for each $i\in Q_0$.
\end{proof}

Let $Q$ be an acyclic quiver with an anti-automorphism $\sigma$ satisfying conditions in Lemma \ref{mm}. Let  $f^t_c(x_i) =  L_i(x_1,\dots,x_n)$,  by Lemma \ref{mm}, we get that
\begin{equation}\label{mmm}
	   f_c^{-t}( x_i) =   f_{\sigma} f_c^tf_{\sigma}(x_i)=  L_{\sigma(i)} (x_{\sigma(1)}, \dots, x_{\sigma(n)}).          \end{equation}
 The action of $\sigma$ on $\mathbb C^n$  is naturally given by $\sigma.(a_1,\dots,a_n) = (a_{\sigma(1)}, \dots, a_{\sigma(n)})$.  It follows from (\ref{mmm})  that  
 \begin{equation}\label{mmmm}
 f_c^{-t}(\mathbf 1)  =  \sigma.(f^t_c(\mathbf 1)), \quad t\in \mathbb Z.
\end{equation}

Examples of acyclic quivers with anti-automorphisms satisfying conditions in Lemma \ref{mm} include affine quivers 
of types $\tilde{\mathbb A}_{p,q}$ and $\tilde{\mathbb D}_n$
we will study in the following section.

We conclude this subsection by  giving a corollary showing that the generalized frieze variety of an acyclic quiver associated to the Coxeter automorphism is the same as that of the opposite quiver associated to its  Coxeter automorphism, which was speculated in \cite{igusa2021frieze}.

\begin{corollary}
Suppose $Q$ is an acyclic quiver and $Q^{\op}$ is the opposite quiver of $Q$.  Let $f_c$ and $g_c$ be the  Coxeter automorphisms of $Q$ and $Q^{\op}$, respectively. If $\mathbf a$ is a general specialization for both $(Q,f_c)$ and $(Q^{op}, g_c)$, then 
\[ X(Q,f_c, \mathbf a) = X(Q^{\op}, g_c, \mathbf a). \]   
\end{corollary} 
\begin{proof}
Assume that the Coxeter mutation of  $Q$ is given by $\mu_c=\mu_n\dots \mu_1$, then  the Coxeter mutation of  $Q^{\op}$ is given by $\mu_c^{op}=\mu_1\dots \mu_n$. Notice that  the cluster algebras $\mathcal{A}(\mathbf x, Q)$ and  $\mathcal{A}(\mathbf x, Q^{op})$ are equal and share the same clusters.  The Coxeter automorphism  $g_c$ of $Q^{op}$ is the inverse of the Coxeter automorphism  $f_c$ of
$Q$.  Thus by Proposition \ref{inverse},
\[     X(Q^{\op}, g_c, \mathbf a) =  X(Q, f_c^{-1}, \mathbf a) =         X(Q, f_c, \mathbf a).               \]
This finishes the proof.
\end{proof}

\subsection{Folding theory}
In this subsection, we work with valued quivers which naturally correspond to skew-symmetrizable matrices, and we refer to section 9 of \cite{keller2013periodicity} and section 4.4 of \cite{fomin2017introduction} for more details about valued quivers and folding techniques.

Firstly,  let us briefly recall some facts about folding theory. Let $Q$ be a quiver and $\Gamma$ a group acting on the vertex set of $Q$.   We say $Q$ is {\em $\Gamma$-admissible}  if  the following conditions are satisfied:
\begin{enumerate}
\item[-] for each $g\in \Gamma$, $g$ induces an automorphism of the quiver $Q$, i.e., for any $i,j\in Q_0$,
                       \[  | i\rightarrow j  | =|g(i) \rightarrow g(j)|,  \]
where $|i\rightarrow j|$ is the number of arrows from $i$ to $j$.
\item[-] there are no arrows among vertices in the same $\Gamma$-orbit.
\item[-]  there are no paths of the form  $i\rightarrow j\rightarrow g(i)$  for $i,j\in Q_0$ and $g\in \Gamma$.
  \end{enumerate}

Let  $Q$ be $\Gamma$-admissible.  For any vertex $i$, the $\Gamma$-orbit of $i$ is denoted by $\Gamma i$, the  mutation sequence  $\mu_{\Gamma i} :=\prod_{i'\in \Gamma i} \mu_{i'}$ is well-defined for $Q$ and is called the {\em orbit mutation}. The {\em folded  (valued) quiver}  $Q^{\Gamma}$ has the vertex set $Q^\Ga_0:=\{\Gamma i, i\in Q_0\}$, and  for any two different vertices $\Gamma i$ and $\Gamma j$,  there is a valued arrow $\Gamma i\xrightarrow{(v_{ij},v_{ji})}\Gamma j$  with values given by  the following formulas:
\[    v_{ij} = \begin{vmatrix}\sum\limits_{i'\in \Gamma i} (|i'\rightarrow j|-|j\rightarrow i'|)\end{vmatrix} ,\quad \text{and}\;\,   v_{ji} = \begin{vmatrix} \sum\limits_{j'\in \Gamma j} (|j'\rightarrow i| -  |i\rightarrow j'|  ) \end{vmatrix}.                 \]

We say  $Q$ is {\em globally foldable with respect to $\Gamma$} if $Q$ is   $\Gamma$-admissible and $\mu_{\Gamma i_k}\dots\mu_{\Gamma i_1}(Q)$ is  $\Gamma$-admissible for any sequence $i_1,\dots, i_k$.  In the rest of this subsection, we always assume that $Q$ is  globally foldable with respect to $\Gamma$.

Let  $(\mathbf x, Q)$ be a seed such that  $Q$ is globally foldable with respect to $\Gamma$,  
and let  $\pi$  the  homomorphism given by 
\[    \begin{split} \pi: \,\,          \mathbb Q[x^{\pm 1}_i, i   \in Q_0] &\longrightarrow      \mathbb Q[x^{\pm 1}_{\Gamma i},  \Gamma i\in Q_0^{\Gamma} ]\\   x^{\pm 1}_i& \longmapsto  x_{\Gamma i}^{\pm 1}.   \end{split}      \]

Each element $g\in \Gamma$ naturally acts on $\mathbb Q[x_i^{\pm 1}, i\in Q_0]$ by $gx_i=x_{g(i)}$, and  we say a seed $(\mathbf x',  Q')$ is {\em $\Gamma$-invariant} if $g.\mathbf x' = \mathbf x'$ for every $g\in \Gamma$ and $Q'$ is $\Gamma$-admissible.  
The following lemma shows that each seed obtained from a sequence of orbit
mutations is $\Gamma$-invariant\footnote{The converse of this assertion can be false. Indeed, the quiver $Q$ in Example \ref{markovunfolding} is globally foldable with respect to  $\Gamma=\langle (14)(25)(36)\rangle$ and the folded quiver is the Markov quiver in Example \ref{markovexa}.  The  cluster  $f(\mathbf x)$ is $\Gamma$-invariant, while if it can be obtained by orbit mutations, then $\pi(f(\mathbf x))$ is a cluster of Markov quiver. This can not happen, because  $(9,9,9)=\pi(f(\mathbf x))|_{x_{\Gamma 1} = x_{\Gamma 2}=x_{\Gamma 3}=1}$ does not belong to the zero locus of $x^2+y^2+z^2-3xyz$.}, and  it   is easy to prove by the induction on the length  of the sequence of orbit mutations.
\begin{lemma}\label{igi}
Let $(\mathbf x', Q')$ be a seed  obtained from  $(\mathbf x,Q)$ by a sequence of  orbit mutations.  Write $x'_i = L_i(x_1,\dots,x_n)$ as a Laurent polynomial  of $\mathbf x$.  Then 
for any  $i\in Q_0$ and $g\in \Gamma$, we have that 
\[      x'_{g(i)} = L_i  (x_{g(1)}, \dots, x_{g(n)}).      \]
In particular, $\pi(x'_i) = \pi(x'_{g(i)})$  for any $i\in Q_0$ and $g\in \Gamma$.
\end{lemma}
 
We  then  define a new seed  $(\mathbf x', Q')^{\Gamma} :=   ({\mathbf x'}^{\Gamma}, {Q'}^{\Gamma})$ associated to each $\Gamma$-invariant seed $(\mathbf x', Q')$ such that \[{\mathbf x'}^{\Gamma} :=  (x'_{\Gamma i}:=\pi(x'_{i}),  \Gamma i \in Q^{\Gamma}_0).\]

\begin{lemma}\label{fold}
If  $Q$ is globally foldable with respect to $\Gamma$, then  for any  sequence $i_1,\dots,i_k$, we have that
\[      (\mu_{\Gamma i_k}\dots\mu_{\Gamma i_1}(\mathbf x, Q) )^{\Gamma} =       \mu_{\Gamma i_k}\dots\mu_{\Gamma i_1}(\mathbf x^{\Gamma}, Q^{\Gamma}).           \]
\end{lemma}

Let $\mathcal A=\mathcal A(\mathbf x,Q)$ and $\mathcal A^\Gamma = \mathcal A(\mathbf x^\Gamma,Q^\Gamma)$ be the corresponding cluster algebras.   Notice that  $\mathcal A^\Gamma\subset \pi(\mathcal A)$.

\begin{definition}
We say a pair $(f,h)$ is a {\em compatible pair of cluster automorphisms} if the following conditions are satisfied:
\begin{enumerate}
\item[-] $f$ is a cluster automorphism of $\mathcal A$ and $h$ is a cluster automorphism of $\mathcal A^\Gamma$;
\item[-] $f(\mathbf x)$ is obtained from $\mathbf x$ by a sequence of  orbit mutations and $f(\mathbf x)^\Gamma  =  h(\mathbf x^\Gamma)$.
\end{enumerate}
\end{definition}

\begin{example}
	Let $Q$ be an acyclic quiver and $\Gamma\subset \Aut(Q)$ is a subgroup.  It is clear  that $Q$ is $\Gamma$-admissible.  Demonet proved in \cite{demonet2008categorification} that $Q$ is globally foldable with respect to $\Gamma$.  Suppose the Coxeter automorphisms for $Q$ and $Q^\Gamma$ are denoted by $f_c$ and $h_c$, respectively. Then $(f_c,h_c)$ is a compatible pair.  Indeed, assume that $ \mu_{\Gamma i_r}\dots\mu_{\Gamma i_1}$ is the Coxeter mutation of $Q^\Gamma$, then each vertex in $\Gamma i_1$ of $Q$ is a sink.   Similarly each vertex in $\Gamma i_{k+1}$ of $\mu_{\Gamma i_{k}}\dots\mu_{\Gamma i_1}(Q)$ is a sink for $1\leq k\leq r-1$.  This implies $\mu_{\Gamma i_r}\dots\mu_{\Gamma i_1}$ is the Coxeter mutation of $Q$, therefore  $f(\mathbf x)^\Gamma  =  h(\mathbf x^\Gamma)$.	
	\end{example}

\begin{lemma}\label{cfh}
Let $(f,h)$ be a compatible pair of cluster automorphisms.  Then  for any $t\geq 0$, we have that
\[      (f^t(\mathbf x))^\Gamma   =  h^t(\mathbf x^\Gamma).         \]
\end{lemma}
\begin{proof}
Since $f(\mathbf x)$ is obtained from $\mathbf x$ by orbit mutations, we may assume that 
$f(\mathbf x) =  \mu_{\centerdot}(\mathbf x)$ with $\mu_{\centerdot}=\mu_{\Gamma i_k} \dots \mu_{\Gamma i_1}$.  Therefore by Lemma \ref{fold}, we have
  \[h(\mathbf x^\Gamma) =   f(\mathbf x)^{\Gamma} = ( \mu_{\centerdot}(\mathbf x))^{\Gamma}  =  \mu_{\centerdot}(\mathbf x^{\Gamma}).\] 
Since $f$ and $h$ are cluster automorphisms, thus 
\[  (f^t(\mathbf x))^\Gamma = ( \mu_{\centerdot}^t(\mathbf x))^{\Gamma} =  \mu_{\centerdot}^t(\mathbf x^{\Gamma}) =  h^t(\mathbf x^\Gamma)  \]
holds for any $t\geq 0$.
\end{proof}

The group $\Gamma$ acts on $\mathbb C^n$ by  $g.\mathbf a =(a_{g(1)},\dots, a_{g(n)})$ for each $g\in \Gamma$ and $\mathbf a= (a_1,\dots,a_n)\in \mathbb C^n$.
 A point $\mathbf a\in \mathbb C^n$ is  {\em $\Gamma$-invariant} if
  $g.\mathbf a =\mathbf a$ for any $g\in \Gamma$, which is equivalent to  the condition  $a_i = a_{g(i)}$ holds for any  $i \in Q_0$ and any $g\in \Gamma$.  For example, $\mathbf 1=(1,\dots,1)$ is always $\Gamma$-invariant.
If  $\mathbf a$ is $\Gamma$-invariant, we define $\mathbf a^\Gamma := ( a_{\Gamma i}, \Gamma i \in Q_0^{\Gamma})$ satisfying  that $a_{\Gamma i} = a_i$ for any $i$. 

\begin{lemma}\label{com}
Let $(f,h)$ be a compatible pair of cluster automorphisms, and $\mathbf a$ be a $\Gamma$-invariant general specialization of $(Q,f)$. Then the following statements hold:
\begin{enumerate}
\item[(i)] $\mathbf a^\Gamma$ is a general specialization of  $(Q^\Gamma, h)$;
  \item[(ii)]  for any $t\geq 0$, $f^t(\mathbf a)$ is $\Gamma$-invariant, and 
 \item[(iii)]  $f^t(\mathbf a)^{\Gamma} =  h^t(\mathbf a^{\Gamma})$.
\end{enumerate}
\end{lemma}
\begin{proof}  Let $\mathbf x_t = f^t(\mathbf x)  =   (x_{1,t},\dots, x_{n,t})$.
	By Lemma \ref{igi},  $x_{i,t}(\mathbf a) = x_{i,t}(g.\mathbf a) = x_{g(i),t}(\mathbf a)$ for any $i$ and $g$, hence $f^t(\mathbf a)$  is  $\Gamma$-invariant.   
	
For (3), it is enough to prove $f^t(\mathbf a)^{\Gamma}=(f^t(\mathbf x))^\Gamma|_{\mathbf x^\Gamma=\mathbf a^\Gamma}$ by Lemma \ref{cfh}.
Let  \[f^t(\mathbf a)^{\Gamma} = (p_{\Gamma i}, \Gamma i\in Q_0^{\Gamma}), \quad \text {and}\,\,\,\, (f^t(\mathbf x))^\Gamma|_{\mathbf x^\Gamma=\mathbf a^\Gamma}= (q_{\Gamma i}, \Gamma i\in Q_0^{\Gamma}).\] 

By definition,  for each $\Gamma i$, then  $p_{\Gamma i} = x_{i,t}|_{\mathbf x=\mathbf a}$ and  $q_{\Gamma i} = \pi(x_{i,t})|_{\mathbf x^\Gamma=\mathbf a^\Gamma}$.
	By Lemma \ref{igi}, we have that 
\[      \pi(x_{i,t})|_{\mathbf x^\Gamma=\mathbf a^\Gamma} = x_{i,t}|_{\mathbf x=\mathbf a}. \]

Thus (3) holds, and (1) follows from (3). 	 
\end{proof}

Let $P =  P_1\sqcup \dots\sqcup P_r$ be a partition of  $Q_0=\{1,\dots,n\}$, and let 
$S\subset \mathbb C^n$ be a nonempty set such that for any  $\mathbf a=(a_1,\dots,a_n)\in S$,  $a_j=a_k$ whenever $j,k\in P_i$ for some $i$. 

Take $j_i\in P_i$ for each $1\leq i\leq r$, let $p: \mathbb C^n\rightarrow \mathbb C^r$ be the projection  such that \[p(x_1,\dots,x_n) :=
(x_{j_1},\dots,x_{j_r}).\]

Define $p(S) := \{p(\mathbf a), \mathbf a\in S\}\in \mathbb C^r$, and
 take a block (lexicographic) order  $>$ on monomials of  $\mathbb C[x_1,\dots,x_n]$  such that the smaller variables are $x_{j_1},\dots,x_{j_r}$.  We also  view polynomials in  $\mathbb C[x_{j_1},\dots, x_{j_r}]$  as polynomials in  $\mathbb C[x_1,\dots,x_n]$.
   Then we have the following lemma.
 
 \begin{lemma}\label{spss}
With the notations above. The  ideals $\I(S)$ and  $\I(p(S))$ are related by the following equations:
\begin{equation} \label{pss} \I(p(S)) = \I(S)\cap C[x_{j_1},\dots, x_{j_r}],\end{equation}
\begin{equation}\label{sps}  \I(S)  = \langle \I(p(S)), \;\,  x_k- x_{j_i} ,\; k\in P_i,\; k\neq j_i;\;\,1\leq i\leq r \rangle.
\end{equation}

Furthermore, the varieties $\overline S$ and $\overline{p(S)}$  are isomorphic, and
$\overline{S}$ is irreducible if and only if   $\overline{p(S)}$ is irreducible.

Fix a block order $>$ such that  the smaller variables are $x_{j_1},\dots,x_{j_r}$. If  $G$ is a Gr\"obner basis of $\I(S)$ with respect to the block order, then $G\cap \mathbb C[x_{j_1},\dots,x_{j_r}]$ is a Gr\"obner basis of $\I(p(S))$.  Conversely, if $\overline{G}$ is a Gr\"obner basis of $\I(p(S))$, then $G$ is a Gr\"obner basis of $\I(S)$ with respect to the block order, where $G$ is given by 
\begin{equation}\label{GG} G= \overline{G}\cup \{     x_k- x_{j_i} ,\; k\in P_i,\; k\neq j_i;\;\,1\leq i\leq r       \}.\end{equation}
\end{lemma}
\begin{proof}
The equation  (\ref{pss}) is obvious.  For the equation (\ref{sps}),  it is clear that 
\[    J=  \langle  \I(p(S)), \;\,   x_k-x_{j_i} ,\; k\in P_i,\; k\neq j_i;\;\,1\leq i\leq r \rangle \subset  \I(S).     \]

For the converse, let $f\in \I(S)$ and $>$ denote the block order such that   the smaller variables are $x_{j_1},\dots,x_{j_r}$. By Lemma \ref{divisionalgo},  $f$ can be written  as
\begin{equation}\label{divisionf}  f =r+ \sum\limits_{ k\in P_i,\,  k\neq j_i; \,1\leq i\leq r  }     c_{kj_i}(x_k-x_{j_i}),    \end{equation}
with $c_{kj_i},r \in \mathbb C[x_1,\dots,x_n]$ and  either  $r=0$ or  each monomial of $r$ is not divisible by  $x_k$ for  $k\notin \{j_1,\dots,j_r\}$.  If $r=0$, then $f\in J$. Otherwise, it is easy to see that
\[r \in \I(S)\cap C[x_{j_1},\dots, x_{j_r}] = \I(p(S)).\]
Therefore $f\in J$ and $I\subset J$.  

Let $\iota: \mathbb C^r\rightarrow\mathbb C^n$ be the inclusion map such that $b_k= x_{j_i}$ for each $k\in P_i$ and $1\leq i\leq r$, here $(b_1,\dots,b_n)=\iota(x_{j_1},\dots,x_{j_r})$.
It follows from (\ref{pss}) and (\ref{sps}) that $\iota$ and $p$ restrict to inverse isomorphisms on $\overline S$ and $\overline{p(S)}$. Consequently, $\overline{S}$ is irreducible if and only if   $\overline{p(S)}$ is irreducible.

The first statement on Gr\"obner bases follows from the elimination on Gr\"obner bases, see Lemma \ref{eligro}. 
It suffices to show that if $\overline{G}$ is a Gr\"obner basis of $\I(p(S))$, then $G$ is a Gr\"obner basis of $\I(S)$ with $G$ given by (\ref{GG}).  We have shown $\I(S)$ is generated by $G$.  

Let $0\neq f\in \I(S)$ and write $f$ as in (\ref{divisionf}). By Lemma \ref{divisionalgo},
the leading monomial  $\Lm(f)$ of $f$  is either $\Lm(r)$ or is divisible by $x_k$ for some 
$k\notin \{j_1,\dots,j_r\}$. Hence $\Lm(f)$ is divisible by $\Lm(g)$ for some $g\in G$ and thus 
$G$ is a Gr\"obner basis of $\I(S)$ with respect to a given block order.
  \end{proof}

By Lemma \ref{com} and Lemma \ref{spss}, we have the following result.

\begin{theorem}\label{folding}
Suppose that the quiver $Q$ is globally foldable with respect to $\Gamma$ and $(f,h)$  is a compatible  pair of cluster automorphisms of $\mathcal A$ and $\mathcal A^\Gamma$.
Let $\mathbf a\in (\mathbb C^*)^n$ be a $\Gamma$-invariant  general  point for $(Q,f)$. Then
the irreducible decomposition of $X(Q,f,\mathbf a)$ is given by
\[X(Q,f,\mathbf a) = \bigcup_{i=0}^{m-1}X(Q,f^m,f^i(\mathbf a))\]   if and only if the irreducible decomposition of  $X(Q^\Gamma,h,\mathbf a^\Gamma)$ is given by 
\[X(Q^\Gamma,h,\mathbf a^\Gamma) = \bigcup_{i=0}^{m-1}X(Q^\Gamma,h^m,h^i(\mathbf a^\Gamma)).\] 

 In particular, they have the same numbers of irreducible components. 
 Furthermore,  the ideals   $I(X(Q,f, \mathbf a))$ (or  $I(X(Q,f^m, f^i(\mathbf a)))$)  and   $I(X(Q^\Gamma,h,\mathbf a^\Gamma))$  (or $I(X(Q^\Gamma,h^m,h^i(\mathbf a^\Gamma)))$ )  and their  Gr\"obner bases (with respect to a given block order)  are related as in Lemma \ref{spss}.  Moreover $X(Q,f^m, f^i(\mathbf a))$ is smooth if and only if 
 $X(Q^\Gamma,h^m,h^i(\mathbf a^\Gamma))$ is smooth.
\end{theorem}

\begin{example}
	Let $Q$ be the quiver given by 
	
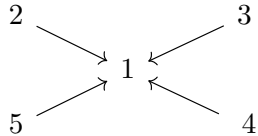
\[\begin{tikzcd}[row sep=0.5em]
	2\arrow[dr]&&3\arrow[dl]\\
	&1&\\
	5\arrow[ur]&&\,\,4\arrow[ul]
	\end{tikzcd}\]

We may take $\Gamma = \langle(45)\rangle, \langle(345)\rangle, \langle(24)(35)\rangle$ or  $\langle (2345)\rangle$.  Let $f_c$ and $h_c$ be the corresponding Coxeter automorphisms for $Q$ and $Q^\Gamma$, respectively.  Then $(f_c,h_c)$ is a compatible pair of cluster automorphisms.

Let $\Gamma = \langle(2345)\rangle$, the folded valued quiver  $Q^\Gamma$ is given by 
\[   \begin{tikzcd}    2 \arrow{r}{(4,\,\,1)} & 1      	\end{tikzcd}            \]

The frieze variety $X(Q^\Gamma)=X(Q^\Gamma,h_c,\mathbf 1)$ is the zero locus of the irreducible polynomial  \[x_2^4-5x_1x_2^2+x_1^2+2x_1+1,\]
 and the frieze variety $X(Q)=X(Q,f_c,\mathbf 1)$ is the zero locus of the polynomials 
\begin{equation}\label{gd4}
	x_2^4-5x_1x_2^2+x_1^2+2x_1+1,\;  x_3-x_2,\; x_4-x_2,\; x_5-x_2,
	\end{equation} 
which form a Gr\"obner basis of the prime ideal $\I(X(Q))$ with respect to the lexicographical order $x_5>x_4>x_3>x_2>x_1$.

	\end{example}

\section{Affine quivers with Coxeter automorphisms}\label{chapteraffine}
In this section, we mainly study generalized frieze varieties associated to the Coxeter automorphism (usually denoted by $f_c$) for  affine quivers.  

 The  {\em Euclidean diagrams} of  type $\widetilde{\mathbb{A}}_{p,q}$ ($p,q\geq 1$), $\widetilde{\mathbb{D}}_{n}$ ($n\geq 5$),  $\widetilde{\mathbb{E}}_{6}$,  $\widetilde{\mathbb{E}}_{7}$, and  $\widetilde{\mathbb E}_{8}$  are the underlying  graphs of the following quivers.

The  affine  quiver of  type $\widetilde{\mathbb{A}}_{p,q}$:

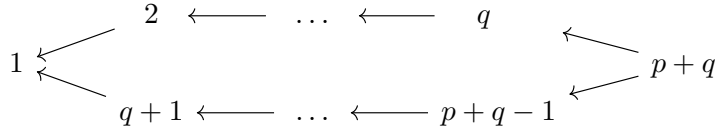
\[\begin{tikzcd}[row sep=0.1em] &\;\;2\;\;\arrow[dl] &\;\;\dots\;\;\arrow[l]&\quad q\quad\quad \arrow[l]\\
	1&&&&p+q \arrow[ul]\arrow[dl]\\
	&q+1\arrow[ul]&\;\dots\;\arrow[l] &p+q-1\arrow[l]&
	\end{tikzcd}\]

The  affine  quiver of  type $\widetilde{\mathbb{D}}_{n}$:

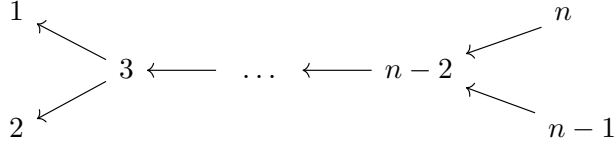
\[\begin{tikzcd}[row sep=0.6em]
	1&&&&n\quad\;\arrow[dl]\\&3\arrow[ul]\arrow[dl]&\;\,\dots\;\,\arrow[l]&n-2\arrow[l] &\\
	2&&&& n-1\arrow[ul]
\end{tikzcd}
\]

The  affine  quiver of  type $\widetilde{\mathbb{E}}_{6}$:

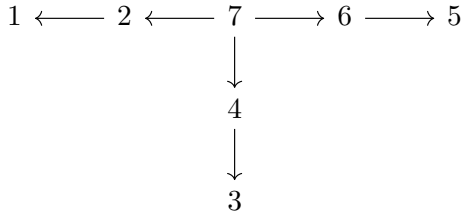
\[\begin{tikzcd}
	1&2\arrow[l] &7\arrow[l]\arrow[r]\arrow[d]& 6\arrow[r]&5\\
	&&4\arrow[d]&&\\
	&&3&&
\end{tikzcd}
\]

The  affine  quiver of  type $\widetilde{\mathbb{E}}_{7}$:

\[\begin{tikzcd}
	1&2\arrow[l] &3\arrow[l]&8\arrow[l]\arrow[r]\arrow[d]& 6\arrow[r]&5\arrow[r]&4\\
	&&&7&&&
\end{tikzcd}
\]

The   affine  quiver of  type $\widetilde{\mathbb{E}}_{8}$:

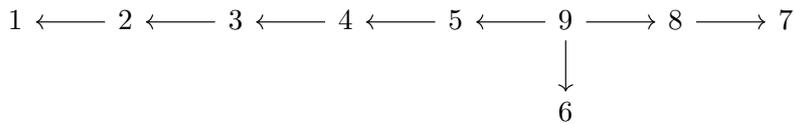
\[\begin{tikzcd}
	1&2\arrow[l] &3\arrow[l] &4\arrow[l] &5\arrow[l] &9\arrow[l]\arrow[r]\arrow[d]& 8\arrow[r]&7\\
	&&&&&6&&
\end{tikzcd}
\]

In this and the following sections, we mainly consider affine quivers with above orientations, and when we say an affine quiver  of type $\Delta$, it is referred to be a quiver of type $\Delta$ with the orientation as  above.  Our main conclusion in this section does not depend on the orientations.

Let $Q$ be an affine quiver with vertex set $Q_0$ and arrow set $Q_1$. The {\em  Euler form} associated to $Q$ is a bilinear form on $\mathbb{Z}^{|Q_0|}$  given by the following formula:
\[\langle u, v\rangle = \sum\limits_{i\in Q_0} u_iv_i - \sum\limits_{\alpha\in Q_1} u_{s(\alpha)}v_{t(\alpha},\]
for $u, v\in \mathbb{Z}^{|Q_0|}$.  The Euler form is not symmetric, we prefer to  use the {\em symmetrized Euler form} defined by \[(u,v) = \langle u, v\rangle  + \langle v, u\rangle . \]

Since $Q$ is affine, the radical of the symmetrized Euler  form is one-dimensional and generated by a positive vector $\delta \in (\mathbb{Z}_{>0})^{|Q_0|}$ such that $\text{min}\{\delta_i,  i\in Q_0\}  =1$.  A vertex $i\in Q_0$ is called an {\em extending vertex}, if $\delta_i=1$.

\begin{theorem}[\cite{keller2011linear, lee2020frieze}] \label{fund}
	Let $Q$ be an affine quiver.  Then there exists  a positive integer $m$ and a Laurent polynomial  $\mathcal{L}\in \mathbb{Q}[x_1^{\pm 1}, \dots, x_n^{\pm 1}]$ such that \[  x_{i,t+2m} -\mathcal{L}\cdot x_{i, t+m} + x_{i,t} =0 \]
	for every extending vertex $i\in Q_0$ and  for all  integers $t\geq 0$.
\end{theorem}
 
	Let $X_{?}$ denote the  Caldero-Chapton map, and $X_{\delta}  :=  X_M$ for some regular simple representation
	$M$ of $Q$ with $\text{dim}M=\delta$.   Dupont  proved $X_{\delta}$ does not depend on the choice of $M$ in \cite{dupont2011generic}, so $X_{\delta}$ is well-defined.
	
	The following table  gives a full list of the integers $m$ and the Laurent polynomials $\mathcal{L}$ for the affine quivers  in Theorem \ref{fund}.
	\begin{table*}[htbp]
		\centering  
		\begin{tabular}{|c|c|c|c|c|c|c|}  
			\hline  
			& & & & &  & \\[-6pt] 
			$Q$ & $\widetilde{\mathbb{A}}_{p,q}$& $\widetilde{\mathbb D}_n$, $n$ is even&  $\widetilde{\mathbb D}_n$, $n$ is odd & $\widetilde{\mathbb E}_6$  & $\widetilde{\mathbb E}_7$ &  $\widetilde{\mathbb E}_8$ \\  
			\hline
			& & & & & & \\[-6pt]  
			$m$ & $\text{lcm}(p,q)$ &  $2(n-3)$ & $n-3$ & $6$ & $12$  &$30$               \\
			\hline
			& & &  & & &\\[-6pt] 
			$\mathcal{L}$& $X_{k(p,q)\delta}$& $X_{\delta}^2-2$ & $X_{\delta}$     &$X_{\delta}$ &  $X_{\delta}$  & $X_{\delta}$            \\  
			\hline
		\end{tabular}
	  \vspace{3mm}
	    \caption{Periods of affine quivers}
	    \label{tab1}
	   \vspace{-1.7em}
	\end{table*}

Here $k(p,q):= (p+q)/\gcd(p,q)$. We call $m$ in Table \ref{tab1} the {\em period} of $Q$.

Theorem \ref{fund}  for $\widetilde{\mathbb{A}}_{p,q}$ was proved in  \cite{lee2020frieze}, see  [Theorem 3.1, \cite{lee2020frieze}].  In \cite{lee2020frieze}, the Laurent polynomial $\mathcal{L}$ is given by the Laurent polynomial associated to the  $k(p,q)$-bracelet of the  equator  in the annulus with
	$p$ marked points on the  inner boundary component and $q$ marked points on the outer boundary component, which equals to   $X_{k(p,q)\delta}$ by Dupont and Thomas, see \cite{dupont2013atomic}. 
	
	The theorem for types $\widetilde{\mathbb{D}}_n,\,   \widetilde{\mathbb{E}}_6, \, \widetilde{\mathbb{E}}_7,$ and $\widetilde{\mathbb{E}}_8$ was proved in  \cite{keller2011linear}, see  [Theorem 6.1,  6.2,  7.1, \cite{lee2020frieze}].

In the rest of this section, we will  prove the following main theorem and  study  .
\begin{theorem}\label{main}
	Let $Q$ be an affine quiver and $\mathbf a$ a general specialization for $Q$ and $f_c$.  Then the generalized frieze variety  $X(Q, \mathbf a) = X(Q, f_c,  \mathbf a)$  is either a finite subset of $\mathbb{C}^n$
	or  a union of finitely many rational curves.   	
\end{theorem}
\begin{proof} By definition, the generalized frieze variety $X = X(Q, f_c, \mathbf a)$  is the Zariski closure of the set  
	\[\{\mathbf a_t =  (a_{1,t}, \dots, a_{n,t}) := \mathbf x_t(\mathbf a)\}_{t\geq 0}\subset \mathbb C^n,\]
	where $\mathbf x_t =  f_c^t(\mathbf x)$ for $t\geq 0$. 
	
	Let $c = \mathcal L(\mathbf a) $. By Theorem \ref{fund},  for each extending vertex $i\in Q_0$,  $0\leq t\leq m-1$ and all $j\geq 0$, we have that 
	\[    a_{i, t+(j+2)m}   -c\cdot a_{i, t+(j+1)m}  + a_{i,t+ jm}=0.      \]

	Thus for each extending vertex $i\in Q_0$ and any $t\in [0, m-1]$,    the sequence $(a_{i,t+jm})_{j\in \mathbb N}$ satisfies a linear recurrence  with the characteristic polynomial 
	\[p(x) =x^2 -c\cdot x +1.\]
	Note that $c$ does not depend on the choices of $i$ and $t$.
	
	 By Lemma \ref{lr3},  if $c\notin\{2,-2\}$,  there exists a complex number $\rho \notin \{1,0,-1\}$  and a  Laurent polynomial  $r_{i,t}(x) = b_{it,-1}x^{-1} + b_{it,1}x\in \mathbb C[x^{\pm 1}]$
	  such that 
	  \[a_{i,t+jm} = r_{i,t}(\rho^j), \quad j\in \mathbb{N}.\] 
	  
	  If $c=2$,  there is a linear polynomial $r_{i,t}(x)= b_{it,0}  +b_{it,1}x  \in \mathbb C[x]$ such that  \[ a_{i,t+jm} = r_{i,t}(j), \quad j\in \mathbb{N}.\]
	  
	   If $c=-2$, there is a linear polynomial $r_{i,t}(x)=b_{it,0}  +b_{it,1}x  \in \mathbb C[x]$ such that  
	  \[a_{i,t+jm} = \begin{cases} r_{i,t}(j), \quad   & j\in 2\mathbb{N},\\  -r_{i,t}(j), \quad & j\in 2\mathbb{N}+1.\end{cases}\]  
	
	Our aim is to show that  there are rational functions $r_{k,t}(x)\in \mathbb C(x)$ satisfying above conditions for each non-extending vertex $k$.  The main idea to  find these rational functions  is  using   almost split triangles for vertices adjacent to extending vertices\footnote{The relations here are indeed exchange relations obtained by a single mutation at the extending vertices from the perspective of cluster algebras.},  and 
	exchange triangles  for others and applying the Caldero-Chapoton map. 
	 All of  these exchange triangles we need are  obtained   in \cite{keller2011linear}.
	
	For $0\leq t\leq m-1$, let  $X_t$ be  the  Zariski closure of the subset
	\[\{\mathbf a_{t+jm}   =   (a_{1,t+jm},  \dots, a_{n, j+tm}) \in\mathbb C^n |\, j\in \mathbb{N} \}\subset \mathbb C^n.\]

	It is obvious that 
	\[    X =      X_0 \cup  X_1 \cup\dots \cup X_{m-1}.   \]

	By Lemma \ref{buqueding}, we have that 
	\[\dim X = \dim X_0  = \dots =\dim X_{m-1}.\]

	Thus  $X$ is a finite subset if and only if there is a $t_0\in [0, m-1]$ such that $X_{t_0}$ is a finite subset, if and only if $X_t$ is a finite subset for all $t\in [0,m-1]$.
	
	In the rest of the proof, we only prove the theorem when $c\notin \{2,-2\}$. The proof for $c\in \{2,-2\}$ is similar.
	
	(1) {\em Type  $\widetilde{\mathbb{A}}_{p,q}$.}
 Let $Q$ be the affine quiver of type  $\widetilde{\mathbb{A}}_{p,q}$ and $n=p+q$. Since every vertex is an extending vertex, we have that  for $t\in [0, m-1]$, there are Laurent polynomials $r_{1,t}(x), r_{2,t}(x), \dots, r_{n,t}(x)\in \mathbb C[x^{\pm 1}]$  and a complex number $\rho \notin \{1,0,-1\}$ such that  for all $j\in \mathbb{N}$, we have that 
	\[\mathbf a_{t+jm} = (r_{1,t}(\rho^j), r_{2,t}(\rho^j), \dots, r_{n,t}(\rho^j)).\]

	If $\rho$ is a root of unity,  $X_t$ is a finite subset and $\dim X_t=0$ for all $t\in [0,m-1]$. If $\rho$ is not a root of unity,  the set $\{\rho^j, j\in\mathbb N\}$ is an infinite set, and by Lemma  \ref{rationalcurve},   $X_t$  is a rational curve for every $t$.  This finishes the proof for the type  $\widetilde{\mathbb{A}}_{p,q}$.

	(2) {\em Type  $\widetilde{\mathbb{D}}_n$.} Let $Q$ be  the affine quiver of type $\widetilde{\mathbb D}_{n}$, and recall  that the extending vertices are $1, 2, n-1$ and $n$. 
	
	For $i\in \{1,2, n-1,n\}$ and $t\in [0,m-1]$, we have a Laurent polynomial  $r_{i,t}(x)$ and a number $\rho \notin \{1,0,-1\}$
	such that $a_{i,t+jm} = r_{i,t}(\rho^j)$  for  $j\in \mathbb{N}$.  We want to prove that for every  vertex $k\in [3, n-2]$, there is also a Laurent polynomial $r_{k,t}(x)$ such that \[ a_{k,t+jm} = r_{k,t}(\rho^j),  \quad j\in \mathbb{N}.\]
	
	   Let us prove it case by case for $3\leq k\leq n-2$ as follows.

	\begin{enumerate}
		\item[(i)] For $k=n-2$, consider the  almost split triangle: 
		\[   \tau^{-t+2}P_n\longrightarrow \tau^{-t+1}P_{n-2} \longrightarrow \tau^{-t+1}P_n\longrightarrow   \Sigma  \tau^{-t+2}P_n .\]
		Then  $x_{n-2, t} =  x_{n,t} x_{n, t-1} -1$. Define
		  \[r_{n-2,t}(x) :=  r_{n,t}(x)r_{n,t+1}(x) -1 \in \mathbb C[x^{\pm 1                                                                                                                                                                                                                                                                                                                                                                       }].\] 
		  We have  for all $j\in \mathbb{N}$, 
		\[a_{n-2, t+jm} = a_{n,t+jm}a_{n,t-1+jm} - 1 = r_{n,t}(\rho^j) r_{n,t+1}(\rho^j) -1  = r_{n-2,t+jm}(\rho^j).\]

		\item[(ii)]  For $k =3$, we have the following  almost split triangle:
		\[ \tau^{-t+1}P_1\longrightarrow \tau^{-t+1}P_{3} \longrightarrow \tau^{-t}P_1\longrightarrow \Sigma \tau^{-t+1}P_1.\]
	Similarly, define 
	\[r_{3,t}(x) :=  r_{1,t}(x)r_{1,t+1}(x) -1 \in \mathbb C[x^{\pm 1}].\] 
		
		\item[(iii)] For $4\leq k \leq n-3$,  we have exchange triangles of $P_{k+1}$ and $S_k$:
	\[  
\ET(P_{k-1},  S_{k})= \begin{cases} 	(P_4, P_1\oplus P_2),\quad   & k=4; \\
	(P_{k}, P_{k-2}), \quad & 5\leq k\leq n-3.
	\end{cases}                \]
		Thus for $t \in \mathbb N$,  we have 
		 \begin{equation}\label{dninv}
		 	\begin{split}               x_{4,t} &= x_{3,t}X_{\tau^{-t+1}S_{4}}  - x_{1,t}x_{2,t},\\    x_{5, t}  &= x_{4,t}X_{\tau^{-t+1}S_{5}} -x_{3,t},\\    &\vdots  \\ x_{n-3, t} & = x_{n-4,t}X_{\tau^{-t+1}S_{n-3}} -x_{n-5,t}.       \end{split}             \end{equation}
			Since for each $4\leq k\leq n-3$,  $\tau^{-t+1}S_{k}  =  \tau^{-t+1-jm}S_{k}$ for all $j\geq 0$,  we may let 
			\[d_{k,t} :=X_{\tau^{-t+1}S_{k}}(\mathbf a).\]
	Then  the following Laurent polynomials are defined inductively.
			   	\[ \begin{split}          r_{4,t} (x)&:= d_{4,t} r_{3, t}(x)  -    r_{1,t}(x) r_{2,t}(x),\\        r_{5, t}(x)  &:= d_{5,t}r_{4,t}(x) -r_{3,t}(x),\\   &\vdots  \\ r_{n-3, t} (x)& :=d_{n-3,t} r_{n-4,t}(x) -r_{n-5,t}(x).       \end{split}             \]
		For $4\leq k\leq n-3$ and $0\leq t\leq m-1$,  it is easy to see that
		\[    a_{k,t+jm} = r_{k,t}(\rho^j),  \quad j \in \mathbb N.                    \]
	Therefore  these rational functions	are what we need.
		 \end{enumerate} 
	 
		Then for each $t\in[0,m-1]$, the variety $X_t$ is parametrized by Laurent polynomials in $\mathbb C[x^{\pm 1}]$, which are  of the form\footnote{By Lemma \ref{lr}, for each non-extending vertex $k$ and each $t$, the sequence $(a_{k, t+jm})_{j\in \mathbb N}$ satisfies a linear recurrence with  characteristic polynomial $(x-\rho^2)(x-1)(x-\rho^{-2})=x^3-bx^2+bx-1$ with $b=\rho^2+\rho^{-2}+1$. Hence $(a_{k,j})_{j\in \mathbb N}$ satisfies the linear recurrence with characteristic polynomial  $x^{3m}-bx^{2m}+bx^m-1$, this is also obtained in [Theorem 3.8, \cite{lee2020frieze}] by using the geometric model of type $\tilde{\mathbb D}_n$.}
		\[\frac{b_2x^2+b_0}{x}\;\,\text{for extending vertices}; \quad \text{and}\;\, \frac{e_4x^4+e_2x^2+e_0}{x^2}\;\, \text{for non-extending vertices}.\]

		Hence  $X_t$  is either a finite subset or a rational curve. 
	
	(3) {\em Type  $\widetilde{\mathbb{E}}_6$.} Let $Q$ be the affine quiver of type $\widetilde{\mathbb{E}}_{6}$.  The extending vertices are $1,3,5$. 
	
	 For $k\in \{2,4,6\}$, we have the almost split 
	triangle:
	\[ \tau^{-t+1}P_{k-1}  \longrightarrow  \tau^{-t+1}P_k \longrightarrow
	\tau^{-t}P_{k-1} \longrightarrow \Sigma\tau^{-t+1}P_{k-1}  .\]
	
	Then for  $t\in\mathbb{N}$,
	\[x_{k, t}  =  x_{k-1,t}x_{k-1,t+1} -1.  \]

Define  \[  r_{k,t}(x) : = r_{k-1,t}(x)r_{k-1, t+1}(x)   -1             \in \mathbb C[x^{\pm 1}].\]  
	
	Let $k=7$.  We have the  exchange triangles:
	\[     \ET(P_1, \tau^{-1}P_2) = (P_7,\tau^{-2} P_1).             \]

	Then we have that 
	\[x_{7,t} = x_{1,t} x_{2,t+1} - x_{1,t+2}.\]

	Similarly, we define
	 \[ r_{7,t}(x) := r_{1,t}(x)r_{2,t+1}(x) - r_{1,t+2}(x).\]

	(4) {\em Type  $\widetilde{\mathbb{E}}_7$.} Let $Q$ be the affine quiver of type $\widetilde{\mathbb{E}}_{7}$.  	
	The extending vertices are $1, 4$.  
	
	For $k=2$, which is adjacent to the extending vertex  $1$, we have the almost split triangle: 
	\[  \tau^{-t+1} P_1\longrightarrow  \tau^{-t+1} P_2 \longrightarrow \tau^{-t}P_1 \longrightarrow \Sigma  \tau^{-t+1} P_1.            \]

	Thus  $x_{2,t} = x_{1,t}x_{1,t+1} -1$, and we define \[r_{2,t}(x) :=  r_{1,t}(x)r_{1,t+1}(x)-1.\]  
	
	Similarly $x_{5,t} = x_{4,t}x_{4,t+1} -1$, and we define \[r_{5,t}(x) :=  r_{4,t}(x)r_{4,t+1}(x)-1.\]
	
	For  $k =3, 6, 8$, we use the exchange triangles:
	\[ \begin{split} &\ET(P_1,  \tau^{-1}P_2) = (P_3, \tau^{-2} P_1), \\ & \ET(P_4,  \tau^{-1}P_5) = (P_6, \tau^{-2} P_4), \\& \ET(P_1,  \tau^{-1}P_3) = (P_8, \tau^{-2} P_2).  \end{split}  \]

	These triangles allow us to define the desired rational functions
	\[   \begin{split}
	         &r_{3,t}(x) := r_{1,t}(x)r_{2,t+1}(x) -r_{1,t+2}(x),  \\
	         &r_{6,t}(x):=   r_{4,t}(x)r_{5,t+1}(x) - r_{4,t+2}(x),  \\
		&r_{8,t}(x) :=  r_{1,t}(x)r_{3,t+1}(x)-r_{2,t+2}(x).
		\end{split}                                    \]

	 Let $k=7$.  In this case, we have the exchange triangle
	 \[     \ET(P_1,\tau^{-4}P_4)  =    (\tau^{-1}P_7, N),\]
	where $N$ is the cokernel of a nonzero morphism $\tau^{-1}P_1\rightarrow \tau^{-4}P_4$. 
Then  
\begin{equation} \label{e7inv}    x_{7,t}=x_{1,t-1}x_{4,t+3} -        X_{\tau^{-t+1}N}.            \end{equation}
 	Note that  \[ X_{\tau^{-(t+jm)}N}  =  X_{\tau^{-t}N},\quad  \text{for all $t$ and $j$}.\]
	
	Let $u_t:= X_{\tau^{-t+2}N}(\mathbf a)$, we therefore obtain the desired rational function 
	\[     r_{7,t}(x) := r_{1,t-1}(x)r_{4,t+3}(x)-u_t.               \]

	(5) {\em Type  $\widetilde{\mathbb{E}}_8$.}  Let $Q$ be the affine quiver of type $\widetilde{\mathbb{E}}_{8}$.   The extending vertex is $1$.
	 
	Similarly, for $k=2$, we use the almost split triangle
	\[     \tau^{-t+1}P_1\longrightarrow        \tau^{-t+1}P_2\longrightarrow \tau^{-t}P_1\longrightarrow  \Sigma\tau^{-t+1}P_1.               \]
	
	Then define $r_{2,t}(x) := r_{1,t}(x)r_{1,t+1}(x)-1$.
	
	For $k=3,4,5,9$, we use the following exchange triangles:
	\[  \begin{split} &\ET(P_1,  \tau^{-1}P_2) = (P_3, \tau^{-2} P_1), \quad \ET(P_1,  \tau^{-1}P_3) = (P_4, \tau^{-2} P_2), \\ &\ET(P_1,  \tau^{-1}P_4) = (P_5, \tau^{-2} P_3), \quad \ET(P_1,  \tau^{-1}P_5) = (P_9, \tau^{-2} P_4).
\end{split}	\]

Similarly, we define 
\[   \begin{split} &  r_ {3,t}(x) := r_{1,t}(x)r_{2,t+1}(x) -r_{1,t+2}(x),\quad            r_ {4,t}(x) := r_{1,t}(x)r_{3,t+1}(x) -r_{2,t+2}(x),       \\   
	   &       r_ {5,t}(x) := r_{1,t}(x)r_{4,t+1}(x) -r_{3,t+2}(x),\quad            r_ {9,t}(x) := r_{1,t}(x)r_{5,t+1}(x) -r_{4,t+2}(x).   \end{split}                    \]
	
For  $k=6,7, 8$, we  use the  exchange triangles:
	\[  \begin{split} &\ET(P_1,  \tau^{-7}P_1) = (\tau^{-2}P_7,  T), \\& \ET(P_1,  \tau^{-3}P_7) = (\tau^{-1}P_6, \tau^{-8} P_1), \\ &\ET(P_1,  \tau^{-2}P_6) = (\tau^{-1}P_8, \tau^{-4} P_7),
\end{split}	\]
where $T$ is a regular simple module satisfying that $\tau^{-t-jm}T = \tau^{-t}T$ for all  $t$ and $j$. In particular, we have  
\begin{equation}\label{e8inv}
x_{7,t} = x_{1,t-2}x_{1,t +5} - X_{\tau^{-t+2} T}.
\end{equation}

 Let $v_t = X_{\tau^{-t+2} T}(\mathbf a)$,  we  then define 
	\[ \begin{split} r_{7,t}(x) &:= r_{1,t-2}(x)r_{1,t +5}(x) - v_t, \\
		r_{6,t}(x) &:= r_{1,t-1}(x)r_{7,t+2}(x) - r_{1,t+7}(x), \\
		r_{8,t}(x) &:= r_{1,t-1}(x) r_{6,t+1}(x) - r_{7,t+3}(x).
	\end{split}\]
	
	These rational functions $r_{i,t}$'s  satisfy that 
	\[    a_{i,t+jm} =r_{i,t}(\rho^j),\quad \text{for $j\in \mathbb N$}.                  \]
	
	This implies $X_t$ is either a finite subset or  a rational curve.
\end{proof}

\begin{remark}
	One obtains invariant Laurent polynomials of period $m$  in the sense of  \cite{igusa2021frieze} from the relations (\ref{dninv}), (\ref{e7inv}), and  (\ref{e8inv}).	
	\end{remark}

As a direct corollary of Lemma \ref{rationalgenus} and Theorem \ref{main},  we have the following result.
\begin{corollary}
	Let $Q$ be an affine quiver.  If $\dim X(Q,f_c,\mathbf a) =1$, 
	then the genus of each irreducible components of $X(Q,f_c,\mathbf a)$ is zero.
\end{corollary}

\begin{corollary}
Let $Q$ be an affine quiver of type $\Delta$ and $\mathbf a\in (\mathbb C^*)^n$ a general specialization for $(Q,f_c)$.  Let $\mathcal L$ be the Laurent polynomial in Table \ref{tab1} and $c=\mathcal L(\mathbf a)$. Suppose that  $c\neq \pm 2$ and  $p(x) =x^2-cx+1= (x-\rho)(x-\frac{1}{\rho})$. Then $\dim X(Q,f_c,\mathbf a) = 1$ if and only if  $\rho$ is not a root of unity.   
\end{corollary}
\begin{proof}
	Let  $X=X(Q,f_c,\mathbf a)$ and  $S= \{\rho^j, j\in \mathbb N\}$.
For  each irreducible component  $Y$ of  $X$, there  is a rational map 
$\varphi: \mathbb C\dashrightarrow \mathbb C^n$ such that $Y= \overline{\varphi(S)}$.
If  $\rho$ is a root  of unity,  then $S$ is finite, and $\dim X=\dim Y= 0$. 
If  $\rho$ is not a root of unity, then $S$ is infinite. Since  $\varphi$ is obviously nontrivial in this case,   by Lemma 	\ref{rationalcurve},  $\dim X=\dim Y=1$.
	\end{proof}

The numbers of irreducible components of frieze varieties are closely related to the   
period $m$ in Table \ref{tab1}.

\begin{corollary}\label{numberofirr}
	Let $Q$ be an affine quiver of type $\Delta$ and $X(Q,\mathbf a)= X(Q,f_c,\mathbf a)$ has the dimension one.
	Then the number of irreducible components of the generalized frieze variety $X(Q,\mathbf a)$ divides the period $m$ (as given in Table \ref{tab1}).
\end{corollary}
\begin{proof}
	By Theorem \ref{main}, the irreducible decomposition of $X(Q,\mathbf a)$ is given by 
	\[X(Q, \mathbf a)  =  X_0(Q,\mathbf a)\cup X_1(Q,\mathbf a) \cup \dots \cup X_{m-1}(Q,\mathbf a),\]
	where $X_t(Q,\mathbf a) = \overline{\{\mathbf x_{t+jm}(\mathbf a), j\in\mathbb N\}}$ for $0\leq t\leq m-1$,  and each variety $X_t(Q, \mathbf a)$ is a rational curve. 
	
	The conclusion follows from the following two observations:
	\begin{enumerate}
		\item[-]\ If $X_i(Q,\mathbf a)\subset X_j(Q,\mathbf a)$, then they are equal. 
		\item[-]\  If $X_i(Q,\mathbf a) = X_{i+k}(Q,\mathbf a)$ for some $i$ and $k$, then $X_t(Q,\mathbf a) = X_{t+k}(Q,\mathbf a)$ for all $t$. Here the subscripts are numbers modulo $m$. 
	\end{enumerate}
	
	The second observation is easy to see from Lemma \ref{buqueding} (2).
\end{proof}

 In particular, Corollary \ref{numberofirr} holds for frieze varieties, i.e., $\mathbf a=\mathbf 1$.  The number of irreducible components for a generalized frieze variety of an affine quiver
 is not necessarily the period $m$ in  Table \ref{tab1} and depends on the orientations of quivers and  specializations of initial seeds. Indeed, we will see there are only two irreducible components for  frieze variety of type $\tilde{\mathbb E}_6$, and  $6$ irreducible components of frieze variety of type $\tilde{\mathbb E}_7$  in the next section.  It seems that the number of irreducible components is closely related to the symmetry of the quiver.

For example,  if the  initial value is given by $\mathbf 1$,  the period $m$ for type $\widetilde{\mathbb D}_n$  will be $n-3$  whenever $n$ is odd or even.  Indeed, let  $\sigma\in S_n$ be given by   $\sigma :=(1,2)(n-1,n)$. Then  for each extending vertex $i\in \{1,2,n-1,n\}$ and $t\in \mathbb N$, we have that
 \[
X_{\delta}X_{i,t+(n-3)} = X_{\sigma(i),t+2(n-3)} + X_{\sigma(i),t}.
 \] 
In this case, $X_{\sigma(i),t}(\mathbf 1)= X_{i, t}(\mathbf 1)$ for all $t\in \mathbb N$, we therefore  have that 
\[       a_{i,t+2(n-3)} -ca_{i,t+n-3}+a_{i,t} =0.                                             \]
The sequence $\{\mathbf a_{i,t} = \mathbf x_{i,t}(\mathbf 1)\}_{t\in \mathbb N}$ satisfies a linear recurrence with characteristic polynomial $p(x) =x^{2(n-3)}-cx^{n-3}+1$ for each extending vertex $i$, where $c=X_{\delta}(\mathbf 1)$.

 Combining Theorem \ref{main} and Lemma \ref{ratioanlimplicit}, we may give an algorithm to compute the defining polynomials for the irreducible components of generalized frieze varieties of affine quivers associated to Coxeter automorphisms. 
\vspace{3mm}

{\bf Step 1}:  Compute the initial  data 
$\{\mathbf x_{t}(\mathbf a),  0\leq t\leq 2m\}$.
\vspace{3.5mm}

{\bf Step 2}: Compute the characteristic  polynomial  $p(x) = x^2-cx+1$ for  extending vertices. One may use $c = \mathcal L(\mathbf a)$ with $\mathcal L$ given in Table \ref{tab1}.
An alternative way in practice to give $c$ is to  solve the equation 
$a_{i,2m} -c \cdot a_{i,m}  + a_{i,0} = 0$  for an extending vertex  $i$.  The roots $\rho$ and $\rho^{-1}$ of  $x^2-c\cdot x +1$ can also be given. 
\vspace{3.5mm}

{\bf Step 3}: For each $t$, we may use the initial data to give the rational  functions (actually Laurent polynomials) in $\mathbb C(x)$:
 \[r_{1,t}(x)  =\frac{f_{1,t}(x)}{ g_{1,t}(x) },\; \dots,\;    r_{n,t}(x)  =\frac{f_{n,t}(x)}{ g_{n,t}(x) } .\]  

To obtain these rational functions, we may firstly compute $r_{i,t}(x)$ for extending vertices by solving some systems of linear equations from the initial data of $i$:
\[     r_{i,t}(\rho^j) =  a_{i, t+jm},  \quad  j=0,1.       \]

Then one may follow the method  from the proof of Theorem \ref{main}  to obtain rational functions for non-extending vertices.   An alternative (but more complicated in practice) way to obtain these rational functions is to solve some systems of linear equations  for all  vertices, but we may need to compute more initial data.
\vspace{3.5mm}

{\bf Step 4}:   Let  $ \widetilde{J}_t =  \langle x_k\,g_{k,t}(x)-f_{k,t}(x), 1\leq k\leq n\rangle  \subset \mathbb C[x, x_1,\dots,x_n].$
Compute the Gr{\"o}bner basis of  $\widetilde J_t$ with the lexicographical order $x>x_1>\dots>x_n$.
Then by Lemma \ref{ratioanlimplicit}, the polynomials in the Gr{\"o}bner basis not involving the indeterminate  $x$  are exactly the Gr\"obner basis of the rational curve $X_t(Q,f_c,\mathbf a)$.
\vspace{3.5mm}

In practice, one may compute Gr{\"o}bner bases   by mathematical software, such as Macaulay $2$ or Maple. 

\begin{remark}
If $\mathbf a=\mathbf 1$, the entries in $\mathbf a_j$ could be very large when $|j|$ is large enough. It is more convenient to  compute $\mathbf a_j$ for $-m\leq j\leq m$ as the initial data to formulate the rational parametrization. In particular, for types $\tilde{\mathbb A}_{p,q}$  and $\tilde{\mathbb D}_n$,  $\mathbf a_j$ $(-1\geq j\geq -m)$ and  $\mathbf a_j$ $(1\leq j\leq m)$  are related via a permutation $\sigma$ by (\ref{mmmm}), where  $\sigma$ is respectively given by
\[\small{\left(\begin{array}{c:ccc:ccc:c}
	1 & 2&\dots &q &q+1&\dots&p+q-1&n\\
	n&q&\dots&2&p+q-1&\dots&q+1&1\\
\end{array}\right)\;\text{for $\tilde{\mathbb A}_{p,q}$}  ,\;\text{and}\;   \left(\begin{array}{cccc}
1 & 2&\dots &n\\
n&n-1&\dots &1\\
\end{array}\right)   \;\text{for $\tilde{\mathbb D}_{n}$}.    }  \]
\end{remark}

\section{Gr{\"o}bner bases for frieze varieties of affine types}\label{groaff}
In this section, we study Gr{\"o}bner bases with a given lexicographical order for frieze varieties $X(Q,f_c,\mathbf 1)$ of affine quivers given in Section \ref{chapteraffine}, and the method works for generalized frieze varieties of affine quivers with arbitrary general  specializations and arbitrary orientations of Euclidean diagrams. 

 We give a uniform method to obtain the Gr\"obner bases for types $\widetilde{\mathbb A}$ and $\widetilde{\mathbb D}$, and we will also give some basic examples to show how the algorithm works. For the exceptional cases, we will see their Gr\"obner bases  look like and the numbers of irreducible components will be given.  In particular, we show that every irreducible component of frieze varieties of affine quivers is smooth.  Before we discuss frieze varieties of affine quivers, let us make some preparations.

Let  $\varphi: \mathbb C^*\rightarrow \mathbb C^n$ be a map given as follows:
\begin{equation}\label{tingting}\begin{split} \varphi: \mathbb  C^*&\longrightarrow\; \mathbb C^n \\
		t\;&\longmapsto \;(r_1(t), r_2(t), \dots, r_{n-2}(t),r_n(t)),\end{split}\end{equation}
where  $r_1(t)$ and  $r_n(t)$ are given by
 \[r_1(t) = \frac{a_1t^2+b_1}{t},\quad \text{and}\,\,\, r_n(t) =\frac{a_nt^2+b_n}{t}, \]
with $a_1,b_1,a_n,b_n\in \mathbb C^*$, and  for $k\in [2,n-1]$,
 $r_k(t)$ is of the following form 
 \[\text{either} \,\,\, r_k(t)= \frac{a_kt^2+b_k}{t},\quad \text{or}\,\,\, r_k(t)=\frac{c_kt^4+d_kt^2+e_k}{t^2},\]
 with $a_k,b_k,c_k,d_k,e_k\in \mathbb C$ . We shall always assume that $r_k(t)$ is in the reduced form, that is, $r_k(t) = a_kt$ if $b_k=0$, or $c_kt^2+d_k$ if $e_k=0$.

\begin{lemma}\label{a1}
	With the notations above.  There exists a nonzero irreducible polynomial $h\in \mathbb C[x,y]$ such that $h(r_1(t),r_n(t))=0$, i.e.,
	\[h(\frac{a_1t^2+b_1}{t},\frac{a_nt^2+b_n}{t})=0.\]
	Moreover,  $\deg(h) \leq 2$ and $\deg(h)=1$ if and only if $\alpha := a_1b_n-a_nb_1=0$.
\end{lemma}
\begin{proof}
	Since  $\overline{\pi_{1n}(\varphi(\mathbb C^*))}$  is a  rational  plane curve, it is defined  by an irreducible polynomial $h\in \mathbb C[x,y]$.  It is obvious that 
	\[\deg(h) =1 \iff \alpha = a_1b_n-a_nb_1=0.\] In this case,  $h=a_nx-a_1y$.
	The polynomial $h$ can  also  be given as follows.
	
	View  $f = tx-a_1t^2-b_1$ and  $g = ty-a_nt^2-b_n$ as polynomials of $t$ with coefficients in  $\mathbb C[x,y]$. Let $H = \mathrm{Res}_t(f,g)$ be the resultant of $f$ and $g$:
	\begin{equation}\label{resultant}\begin{split} H &= \det{\begin{pmatrix} -a_1&0&-a_n&0\\x&-a_1&y&-a_n\\-b_1&x&-b_n&y\\0&-b_1  &0&-b_n \end{pmatrix}} \\ 
			&=a_nb_nx^2	- (a_1b_n+a_nb_1)xy+a_1b_1y^2 + \alpha^2.
	\end{split}\end{equation}
	
	Then $H = h^k$ for some positive  integer $k$. Thus $\deg(h)\leq 2$.   Moreover, if $\alpha \neq 0$,  $H$ is irreducible and $H=h$.
\end{proof}

In the following  discussion, we  assume that $\alpha := a_1b_n-a_nb_1\neq 0$.

\begin{lemma}\label{a2}
	With the notations above.	For any $k\in [2,n-1]$, there is a polynomial $\tilde L_k\in\mathbb C[x,y]$ of degree $\leq 2$ such that $r_k(t) = \tilde L_k(r_1(t), r_n(t))$. 
\end{lemma}
\begin{proof}
Suppose that
\[r_k(t) = \frac{a_kt^2+b_k}{t}.\]
Then  it is easy to see that  $\tilde L_k$ is given by 
	\begin{equation} \label{llk}   \tilde L_k=  -\frac{a_nb_k -b_na_k}{\alpha}x  +\frac{a_1b_k-b_1a_k}{\alpha}y.                              \end{equation}

In this case, the polynomial $\tilde L_k$ given in (\ref{llk}) is the unique linear polynomial satisfying the condition. 

Assume that  \[r_k(t) = \frac{c_kt^4+d_kt^2+e_k}{t^2}.\]  

To find $\tilde L_k(x,y)\in \mathbb C[x,y]$ of degree $\leq 2$ such that 
\[ \frac{c_kt^4+d_kt^2+e_k}{t^2}  =  \tilde L_k(\frac{a_1t^2+b_1}{t},\frac{a_nt^2+b_n}{t}),\]
it suffices to assume that $\tilde L_k=  m_1x^2+m_2y^2+ m_3xy- m$. By solving a system of
linear equations,  we have that 
\[\begin{split}
m_1&=\frac{a_n^2e_k-a_nb_n(m+d_k) + b_n^2c_k}{\alpha^2},\\
m_2&=\frac{a_1^2e_k-a_1b_1(m+d_k) +b_1^2c_k}{\alpha^2},\\
m_3&=\frac{-2a_1a_ne_k+(a_1b_n+a_nb_1)(m+d_k)-2b_1b_nc_k}{\alpha^2}, 
\end{split}\]
with an arbitrary complex number $m\in \mathbb C$.  Thus in this case, there is a family of   polynomials $\tilde L_k[m]$ of degree $\leq 2$, which is parametrized by $\mathbb C$, satisfying the condition. Among them, we are only concerned about  the following one:
\begin{equation}\label{2lk}
		\tilde L_k =   -\frac{a_1a_ne_k-b_1b_nc_k}{a_1b_1\alpha}x^2+\frac{a_1^2e_k-b_1^2c_k}{a_1b_1\alpha}xy-\frac{a_1^2e_k-a_1b_1d_k+b_1^2c_k}{a_1b_1}.
\end{equation}
This finishes the proof.
\end{proof}

For  $k\in [2,n-1]$, we define
 \begin{equation}\label{lk}
 	L_k = x_k - \tilde L_k(x_1,x_n)  \in \mathbb C[x_1,\dots,x_n],
 \end{equation}
  here $\tilde L_k$ is given either by (\ref{llk}) when $r_k=(a_kt^2+b_k)/t$ or by (\ref{2lk}) when $r_k=(c_kt^4+d_kt^2+e_k)/t^2$.  We also view $L_k$  defined  in (\ref{lk}) for $2\leq k\leq n-1$ and  $H(x_1,x_n)$ defined in (\ref{resultant})  as polynomials in $\mathbb C[t,x_1,\dots,x_n]$.  Let $L\in \mathbb C[t,x_1,\dots,x_n]$ denote the linear polynomial given by 
\begin{equation}\label{lt}  L= t    -\frac{b_n}{\alpha}x_{1}   +  \frac{b_{1}}{\alpha}x_n.\end{equation}

\begin{theorem}\label{Groebnertypea}
		Let $\varphi$ be given as in (\ref{tingting}) such that  $\alpha = a_{1}b_n- a_nb_{1}\neq 0$, and let  $H$,   $L$ and  $L_k\, (2\leq k\leq n-1)$ the polynomials given  as in (\ref{resultant}), (\ref{lt}) and  (\ref{lk}) respectively. 
		Then  $\{L, L_2,\dots,L_{n-1}, H\}$ is the Gr{\"o}bner basis of the ideal 
		\[\widetilde J =  \langle x_ig_i(t)-f_i(t), 1\leq i\leq n  \rangle \subset \mathbb C[t,x_1,\dots,x_n]\]  
	    with respect to the monomial order $t>x_2>\dots>x_ {n-1}> x_n>x_1$,	where $f_i(t)$ is the numerator of $r_i(t)$ and $g_i(t)$ is the denominator of $r_i(t)$ for $1\leq i\leq n$.

		Hence $\{ L_2,\dots,L_{n-1},H\}$ is the Gr{\"o}bner basis for $\overline{\varphi(\mathbb C^*)}$ with respect to the monomial order $x_2>\dots>x_ {n-1}> x_n>x_1$.  
	\end{theorem}
\begin{proof}
		Let $\tilde G=\langle  L, L_1,\dots,L_{n-1}, H\rangle$. It is easy to see that 
	$\Z(\widetilde J)\subset \Z(\tilde G)$.  Since $\widetilde J$ is prime by  Lemma \ref{ratioanlimplicit},
	therefore $\tilde G\subset \widetilde J$. To see $\widetilde J\subset \tilde G$, it  easily follows from the following facts: 
	\[ \begin{split}
	tx_{1}-a_{1}t^2-b_{1} &= -\frac{b_{1}}{\alpha^2}H+\frac{-\alpha a_{1}t+b_{1}(a_{1}x_n-a_nx_{1})}{\alpha}L,\\	
		tx_n-a_nt^2-b_n &= -\frac{b_n}{\alpha^2}H+\frac{-\alpha a_{n}t+b_n(a_{n-1}x_n-a_nx_{n-1})}{\alpha}L.
 	\end{split}\]	

	If  $r_k= (a_kt^2+b_k)/t$, then 
	\[	tx_k-a_kt^2-b_k = -\frac{b_k}{\alpha^2}H+tL_k+
	\frac{-\alpha a_{k}t+b_k(a_{n-1}x_n-a_nx_{n-1})}{\alpha}L.\]

	If $r_k=(c_kt^4+d_kt^2+e_k)/t^2$, then
	\[\begin{split} t^2x_k-c_kt^4-d_kt^2-e_k=& -\frac{b_1^2c_kt^2+a_1e_kxt+a_1b_1e_k}{a_1b_1\alpha^2}H + t^2 L_k+(\frac{(a_1x_n-a_nx_1)e_k}{\alpha}\\
		&+\frac{(-a_nx_1^2+a_1x_1x_n-a_1\alpha)e_kt + (-\alpha t+b_1x_n-b_nx_1)b_1c_k t^2}{b_1\alpha})L.
		\end{split}\]

	Hence $\tilde J =\langle L, L_2,\dots, L_{n-1}, H \rangle$.

	Let us prove that $G= \{ L, L_2,\dots, L_{n-1}, H \}$ is a Gr{\"o}bner basis for $\widetilde{J}$ with respect to the monomial order $t>x_2>\dots>x_n>x_1$. By Buchberger's Criterion,  it is equivalent to show the remainders of $S$-polynomials  $S(p,q)$ for all $p\neq q\in G$ on division of $(L,L_2,\dots, L_{n-2}, H)$ are zeros. Indeed, we have that
 	have the following decompositions:
	\[S(L,H)=\frac{b_{1}x_n-b_nx_{1}}{\alpha}H-(a_{1}a_nx_n^2-(a_{1}b_n+a_nb_{1})x_1x_n+\alpha^2)L.\]
	
	 For any $i, k\in[2,n-1]$ with $r_i=(a_it^2+b_i)/t,  r_k=(a_kt^2+b_k)/t$ and $k<i$,   and for $s,l\in [1,n]$,	let \[A_{sl}:=	a_sb_l-b_sa_l,\]
	 we have that
		\[\begin{split} 
		S(L,L_k)&= \frac{b_{1}x_n-b_nx_{1}}{\alpha}L_k-\frac{A_{k,1}x_n-A_{k,n}x_{1}}{\alpha}L,\\
		S(L_k,L_i)&=\frac{A_{k,1}x_n-A_{k,n}x_{1}}{\alpha}L_i-\frac{A_{i,1}x_n-A_{i,n}x_{1}}{\alpha}L_k,\\
		S(L_ k,H) &=\frac{A_{k,1}x_n-A_{k,n}x_{1}}{\alpha}H-(a_{1}b_{1}x_n^2-(a_{1}b_n+a_nb_{1})x_{1}x_n+\alpha^2)L_k.
	\end{split}\]

For $i,k\in[2,n-1]$ with $r_i=(c_it^4+d_it^2+e_i)/t^2, r_k =(c_kt^4+d_kt^2+e_k)/t^2$ and $k<i$, let 
\[        C_{ki} := e_kc_i-e_ic_k,          \]
we have that
\[\begin{split}
	S(L,L_k)=&\frac{(b_1b_ nc_k-a_1a_ne_k)x_1^2+(a_1^2e_k-b_1^2c_k)x_1x_n-(b_1^2c_k-a_1b_1d_k+a_1^2e_k)\alpha}{a_1b_1\alpha}L\\
	&+ \frac{b_1x_n-b_nx_1}{\alpha}L_k,
	  \\
	S(L_k,L_i)=&\frac{C_{ki}(tx_1+b_1)}{b_1\alpha}H-c_it^2L_k+c_kt^2L_i\\
	&+\frac{a_ntx_1^2 -a_1txy+a_nb_1x_1-a_1b_1x_n+a_1\alpha t }{b_1\alpha}C_{ki}L,
	\\
S(L_k,H) = &-(a_nb_nx_1^2 -(a_1b_n+a_nb_1)x_1x_n +\alpha^2)L_k\\ &+\frac{(a_1a_ne_k-b_1b_nc_k)x_1^2+(b_1^2c_k-a_1^2e_k)x_1x_n+(b_1^2c_k-a_1b_1d_k+a_1^2e_k)\alpha}{a_1b_1\alpha}H,
\end{split}\]

Finally, if $i,k\in[2,n-1]$ with $r_i=(a_it^2+b_i)/t,  r_k =(c_kt^4+d_kt^2+e_k)/t^2$ and $k<i$ ($i<k$ is similar), we have that
\[ \begin{split} S(L_k,L_i) =& \frac{-A_{1i}b_1c_kt^2 +a_1a_ie_k(tx+b_1) }{a_1b_1\alpha^2}H - a_it^2L_k +c_kt^3L_i \\
&	+ \frac{(A_{1i}x_n+A_{in}x_1)b_1c_kt^2 +((a_nx_1^2-a_1x_1x_n +a_1\alpha )t+b_1a_nx_1-a_1b_1x_n)a_ie_k}{b_1\alpha}L
\end{split} \]

Thus remainders of $S$-polynomials with respect to 
	$(L,L_1,\dots,L_{n-1},H)$ and the   order $t>x_2>\dots>x_n>x_1$
	are zeros.  Therefore $G$ is a Gr\"obner basis of $\tilde J$ with respect to the  order $t>x_2>\dots>x_n>x_1$, and hence $\{L_2,\dots,L_{n-1},H\}$ is a Gr\"obner basis of 
	$J = \I(\overline{\varphi(\mathbb C^*)}) =\tilde J\cap \mathbb C[x_1,\dots,x_n]$.
	\end{proof}

The following result  about smoothness follows easily from  Theorem \ref{Groebnertypea}.

\begin{corollary} \label{smoothad}
		Let $\varphi$ be given as in (\ref{tingting}) such that  $\alpha = a_{1}b_n- a_nb_{1}\neq 0$.  Then the irreducible variety  $\overline{\varphi(\mathbb C^*)}$ is isomorphic to a plane conic and hence it is smooth.
\end{corollary}
\begin{proof}
By Theorem \ref{Groebnertypea}, we  may obtain a Gr\"obner basis  $\{L_2,\dots,L_{n-1}, H\}$  for the prime ideal  $J= \I( \overline{\varphi(\mathbb C^*)})$  such that 
$L_k$ is a polynomial  of the form  $x_k - \tilde L_k(x_1,x_n)$  with $\tilde L_k\in \mathbb C[x,y]$ for $k\in [2,n-1]$ and $H$ is a polynomial   given by 
\[    H  =a_nb_nx_1^2	- (a_1b_n+a_nb_1)x_1x_n+a_1b_1x_n^2 + \alpha^2.      \]

It is obvious that  $\overline{\varphi(\mathbb C^*)}$   is isomorphic to the  irreducible plane curve
$\Z(H(x,y)) \subset \mathbb C^2$. Since $\Z(H(x,y))$  is an irreducible conic which is always smooth, then  $\overline{\varphi(\mathbb C^*)}$ is smooth.
\end{proof}

For the exceptional cases  $\widetilde{\mathbb E}_{n-1}$  ($n=7, 8, 9$), let $m$ be  the period and $\pi_{12}: \mathbb C^n\rightarrow \mathbb C^2$  be the projection mapping $(a_1,\dots, a_n)$ to $(a_1,a_2)$. We will see that  for each $k\geq 0$, $\pi_{12}(\mathbf a_{k+jm})$ lies on a plane curve $\Z(h_0(x,y))\subset \mathbb C^2$ for $j\in \mathbb N$, here $h_0(x,y)$ is an irreducible polynomial of the form 
\[   h_0(x, y) = x^2(ax^2 +by+c) +d(y+1)^2 \]
with  $abcd\neq 0$ and $b^2\neq 4ad$.  Note that $h$ is irreducible, thus $b\neq c$. As we will see later,  for $i \neq 1,2,7$ when $n=7$ and $i\neq 1,2,3$ when $n=8,9$, there is only one  polynomial in  the Gr\"obner basis (with the given monomial order in the following subsections)  involving $x_i$, which  is   of the form 
\[  L_i = a_i x_i +p_i(x_1,x_2),  \quad \text{for}\;\; n=7;\quad  \text{or}\;\;    L_i = a_i x_i +p_i(x_1,x_2,x_3),   \quad \text{for}\;\; n=8,9.  \]

It is easy to see that $\Z(h_0)$, viewed as a plane affine curve in $\mathbb C^2$, has only one singular point $(0,-1)$.  If  $p=(p_1,\dots,p_n)\in X_k$ such that 
$p_1\neq 0$, it is easy to check $p$ is not singular. While $(p_1,p_2)=(0,-1)$, other $p_i$'s are almost determined by $\pi_{12}(p)$, it is also not hard to check $p$ is nonsingular.

\subsection{Type $\tilde{\mathbb  A}_{p,q}$}
In this subsection, we study Gr{\"o}bner bases of  irreducible components of the frieze variety $X(Q)$ of type $\widetilde{\mathbb  A}_{p,q}$.

Let $(a_j)_{j\in \mathbb N} \subset \mathbb Z_{>0}$ satisfy a linear recurrence with the characteristic polynomial $p(x) = x^2 -cx+1$ such that $c\in \mathbb Z\setminus \{\pm 2,0,\pm 1\}$, and let $\rho$ and $\frac{1}{\rho}$  the roots of  $p(x)$. Clearly the nonzero number $\rho$ is not a root of unity.
Let $r(t)$ be the following Laurent  polynomial:
\[r(t) = \frac{\rho(a_0\rho-a_1)}{\rho^2-1}\cdot\frac{1}{t} + \frac{a_1\rho-a_0}{\rho^2-1}\cdot t =\frac{(a_1\rho-a_0)t^2+\rho (a_0\rho -a_1)}{(\rho^2-1)t}.\]

Then we have $a_j = r(\rho^j)$ for any $j\geq 0$.  We also claim that 
$(a_0\rho-a_1)(a_1\rho-a_0)\neq 0$. Indeed, if this happens, we have that $a_j = a_0\cdot(\frac{a_1}{a_0})^j$ for $j\geq 0$, and $\frac{a_1}{a_0}\in \mathbb Z_{>1}$. Since $a_2 -c\cdot a_1 +a_0 = 0$, we have  that \begin{equation} \label{bukeneng} (\frac{a_1}{a_0})^2   +1  = c\cdot\frac{a_1}{a_0}.  \end{equation}
Since $\frac{a_1}{a_0}\in \mathbb Z_{>1}$ and $c\in\mathbb Z$, the equation (\ref{bukeneng}) does not hold, which is a contradiction. Thus for each extending vertex of an affine quiver in Section \ref{chapteraffine}, the corresponding  rational function $r(t)$ is of the form $(at^2+b)/t$ with $ab\neq 0$.

Let $Q$ be the affine quiver of type $\widetilde{\mathbb  A}_{p,q}$.  For any irreducible component of frieze variety $X(Q)$, it has the following parametrizations:
\[
\begin{split} \varphi: \mathbb  C^*&\longrightarrow\; \mathbb C^n \\
		t\;&\longmapsto \;(r_1(t), \dots, r_{n}(t)),\end{split}\]
with $r_i(t) = \frac{a_it^2+b_i}{t}$ and $a_i,b_i\in \mathbb  C^*$ for $1\leq i\leq n$. 

If  for any $i,j$, $a_ib_j-a_jb_i = 0$ holds, then it is clear that  $\overline{\varphi(\mathbb C^*)}$   is isomorphic to  $C$. Otherwise, we may assume that $a_1b_n-a_nb_1 \neq  0$, then
Theorem \ref{Groebnertypea} provides a method to find the Gr\"obner basis of this parametrization, and  Corollary \ref{smoothad} shows this variety is smooth.

\begin{theorem}
	Each irreducible component of frieze variety $X(Q)$ is  isomorphic to  either $\mathbb C$ or a plane conic.  Moreover, if it is isomorphic to a plane conic,	there is a Gr\"obner basis $\{L_k, 3\leq k\leq n; H\}$  of the irreducible component of frieze variety  such that  $L_k$ is given as in (\ref{lk}) and $H$ is given as in (\ref{resultant}). In particular every irreducible component is smooth.	
\end{theorem}

\begin{example}
	The quiver $Q$ of type $\widetilde{\mathbb A}_{2,1}$ is given by 
	\[\mathord{\begin{tikzpicture}[scale=1.3,baseline=0]
			\node at (0,0.3) (12) {$2$};
			\node at (-1,-0.3) (11) {$1$};
			\node at (1,-0.3) (13) {$3$,};
			\path[-angle 90]
			(12) edge   (11)
			(13) edge   (12)
			(13) edge   (11);
	\end{tikzpicture}}\]

	Notice that the period $m=\mathrm{lcm}(1,2) = 2$, and the first five points are given by \[(1,1,1),\, (2,3,7), \,(11,26,41),\, (97,153, 362),\,(571, 1351,2131).\]	
	
	Let us give the defining ideal for $X_0$. The linear polynomial $L_1 = x_1+c_2x_2+c_3x_3$ and the  polynomial $H = b_1x_2^2+b_2x_2x_3+b_3x_3^2+1$ are respectively given by  
	\[
	\begin{pmatrix}  1 &1 &1 \\11 &26 &41
	\end{pmatrix} \begin{pmatrix} 1\\c_2\\c_3\end{pmatrix} = 0,\quad\text{and}\;\,
	\begin{pmatrix} 
		1      &     1       &    1\\
		676     &   1066  &      1681\\
		1825201  &   2878981  &   4541161
	\end{pmatrix} \begin{pmatrix} b_1\\b_2\\b_3\end{pmatrix} = \begin{pmatrix} -1\\-1\\-1\end{pmatrix}. 
	\]
	Thus $L_1= x_1-2x_2+x_3$ and $H = 2x_2^2-6x_2x_3+3x_3^2+1$.  By Theorem \ref{Groebnertypea}, $\I(X_0) = \langle L_1, H\rangle$.
	
	We can also find the rational parametrizations for  $X_0$ and $X_1$. Indeed, the characteristic polynomial is  $p(x)=x^2 - c\cdot x+1$ and $c$ satisfies that $571-c\cdot 11 +1 =0$,
	it is clear that $c= 52$ and two roots of $p(x)$ are  given  by  \[\rho_0=26-15\sqrt{3},\quad \text{and}\;\; \rho_1 = 26+15\sqrt{3}.\]

	To obtain the rational functions of the first component $X_0$, it is enough to solve the  following systems of linear equations:
	\[A \cdot \mathbf b_i = \mathbf c_i,\quad 1\leq i\leq 3,\]
	where 
	\[A=\begin{pmatrix} 1&1\\  \rho_0& \rho_1
	\end{pmatrix},
	\quad \mathbf c_1 = \begin{pmatrix} 1\\11\end{pmatrix} , \quad \mathbf c_2= \begin{pmatrix} 1\\26\end{pmatrix}, \quad \mathbf c_3 = \begin{pmatrix} 1\\41\end{pmatrix}.\]
	Thus we have the rational parametrizations:
	\[(\frac{(3-\sqrt{3})t^2 + 3+\sqrt{3}}{6t}, \,\, \frac{t^2+1}{2t},\,\, \frac{(3+\sqrt{3})t^2 + 3-\sqrt{3}}{6t}). \]

	Let $\tilde{J}_0 =  \langle  6tx_1- (3-\sqrt{3})t^2  - 3-\sqrt{3}),\; 2tx_2 -t^2-1,\; 6tx_3-   (3+\sqrt{3})t^2 -3+\sqrt{3}             \rangle.$	
	The Gr{\"o}bner basis of $\tilde{J}_0$ with respect to $t>x_1>x_2>x_3$ is given by 
	\[2 x_2^{2}-6 x_2 x_3 +3 x_3^{2}+1,\quad x_1-2x_2+x_3, \quad \sqrt{3}x_2-\sqrt{3}x_3-x_2+t.\]
	Then  $\I(X_0) = \langle x_1-2x_2+x_3,  2 x_2^{2}-6 x_2 x_3 +3 x_3^{2}+1\rangle$.
	
	Similarly, we obtain rational parametrizations of the second component $X_1$:
	\[(\frac{(2+\sqrt{3})t^2 +2-\sqrt{3}}{2t},\, \, \frac{(9+5\sqrt{3})t^2+9-5\sqrt{3}}{6t},\, \,\frac{(7+4\sqrt{3})t^2+7-4\sqrt{3} }{2t}).\]

Similarly define $\tilde{J}_1:= \langle     2tx_1- (2+\sqrt{3})t^2 -2+\sqrt{3},\; 6tx_2- (9+5\sqrt{3})t^2-9+5\sqrt{3},\;    2tx_3-(7+4\sqrt{3})t^2-7+4\sqrt{3}  \rangle$, and the Gr{\"o}bner basis of $\tilde{J}_1$ with respect to $t>x_1>x_2>x_3$  is given by 
	\[3x_2^2-6x_2x_3+2x_3^2+1,\quad x_1-3x_2+x_3,\quad 7\sqrt{3}x_2-3\sqrt{3}x_3-12x_2+5x_3+t.\]
	Thus $\I(X_1) = \langle 3x_2^2-6x_2x_3+2x_3^2+1, x_1-3x_2+x_3\rangle$.

\end{example}
In the case of type $\widetilde{\mathbb A}_{1,2}$, the defining polynomials for $X_0$ and  $X_1$ were also obtained by 
Lee,   Li,  Mills,  Schiﬄer and Seceleanu in \cite{lee2020frieze} via Macaulay $2$.  Igusa and Schiffler also gave another method to obtain the defining polynomials for type $\widetilde{\mathbb A}_{1,n}$ by  using an invariant Laurent polynomial,  see   \cite{igusa2021frieze} for details.

The following example shows generalized frieze varieties with different general specializations for the same quiver may have different  numbers of  irreducible components. 

\begin{example}
 Let  $Q$ be the quiver of  type $\tilde{\mathbb A}_{2,2}$ given by 
	\[\begin{tikzcd}
		2\arrow[d] &4\arrow[l] \arrow[d]\\
		1&3\arrow[l]	
	\end{tikzcd}	\]
	
	Since $m=\mathrm{lcm}(2,2)=2$, one may expect there are two irreducible components of  the frieze variety.  The initial five points are given by 
	\[(1,1,1,1),\,(2,3,3,10),\,(5,17,17,29),\\(58,99,99,338),\,(169,577,577,985)\dots \]	
\end{example}

Then $\I(X_0) = \langle  x_1-2x_3+x_4,x_2-x_3, 2x_4^2-4x_3x_4+x_3^2+1 \rangle$ and $\I(X_1) = \langle  x_1-4x_3+x_4,x_2-x_3, x_4^2-4x_3x_4+2x_3^2+2\rangle$.

Let $\mathbf a= (1,1,1,2)$. The generalized frieze variety $X(Q,\mathbf a)$ has only one irreducible component and its corresponding ideal is given by $\I(X(Q,\mathbf a)) = \langle x_3^2-3x_3x_4+x_4^2+1,x_2-x_3,x_1-3x_3+x_4\rangle$. 

Let $Q'=\mu_4(Q)$.  It is an unfolding of the Kronecker quiver.  Even though $Q$ and  $Q'$ have the same underlying graph,  the frieze variety  $X(Q', \mathbf 1)$ has only one irreducible component  whose defining ideal is given by  $\I(X(Q',\mathbf 1)) = \langle x_3^2-3x_3x_4+x_4^2+1,x_2-x_3,x_1-x_4\rangle$.

\subsection{Type $\tilde{\mathbb D}_n$}  
Let $Q$ be the  affine  quiver of type $\tilde{\mathbb D}_n$ given in Section \ref{chapteraffine}, and let $m$ the period. Since we consider the 
frieze variety $X(Q)$, the period $m$ is $n-3$ whenever $n$ is odd or even as we have explained in the previous section.  

 Let $i$ be extending, then $(a_{i, jm}: = \mathbf x_{i,jm}(\mathbf a))_{j\in \mathbb N}$ satisfies   the linear  recurrences with characteristic polynomial $p(x)= x^2-cx+1$ and $c>2$. We write $p(x)$ as the following form 
\[p(x) = (x-\rho^{-1}) (x-\rho)\]
 with $\rho\in \mathbb C^*$   not a root of unity.  
 
 If $k$ is not extending, then
for each $0\leq l\leq m-1$, the sequence  $(a_{k, l+jm})_{j\in \mathbb N}$ satisfies a linear recurrence with the characteristic polynomial 
 given by 
\[
(x-\rho^{-2})(x-1)(x-\rho^2).
\]

In this case,  the irreducible component of frieze variety $X(Q)$ is parametrized by the following map:
\[
\begin{split} \varphi: \mathbb  C^*&\longrightarrow\; \mathbb C^n \\
	t\;&\longmapsto \;(r_1(t), \dots, r_{n}(t)),\end{split}\]
with $r_i(t) = \frac{a_it^2+b_i}{t}$ and $a_i,b_i\in \mathbb  C^*$ for $ i\in \{1,2,n-1,n\}$ and 
$r_k(t) = \frac{c_kt^4+d_kt^2+e_k}{t^2}$ for $k\in [3,n-2]$.  We also have that 
$\alpha = a_1b_n-a_nb_1\neq 0$.  Indeed, in this case, $m=n-3$ and $\mathbf a=\mathbf 1$, we  obtain the initial $2m+1 (=2(n-3)+1) $ data for vertices $1$ and $n$ as follows:
\[  \begin{matrix} ((n-3)n+1, n-2)  &\dots&  ((n-3-j)n+1  ,n-2-j)&\dots    &(n+1, 2)\\
(1,1) &\dots &(j+1, jn+1)& \dots &(n-3, (n-4)n+1),\\
(n-2,(n-3)n+1)&\dots&&&
\end{matrix}
\]
where  the first row denotes $\pi_{1n}(\mathbf a_j)$ for $-m \leq j \leq -1$ which are obtained by applying $f^{-j}_c$, the second row denotes $\pi_{1n}(\mathbf a_j)$ for $0\leq j \leq m-1$,  and the third row denotes $\pi_{1n}(\mathbf a_m)$,  here  $\pi_{1n}:\mathbb C^n\rightarrow \mathbb C^2$ is the projection mapping $(a_1,\dots,a_n)$ to $(a_1,a_n)$.  

For each extending vertex  $i$,   $(a_{i,k+jm})_{j\in \mathbb N}$ satisfies the linear recurrence with the following characteristic polynomial 
\[     x^2-  (n^2-2n-1)x+1.               \]

The following theorem follows from Corollary \ref{smoothad}.

\begin{theorem}\label{smoothd}
	There is a Gr\"obner basis $\{L_k, 2\leq k\leq n-1; H\}$  of each irreducible component of frieze variety of type $\widetilde{\mathbb D}_{n}$ such that  $L_k$ is given as in (\ref{lk}) for $k\in[3,n-2]$
	and $H$ is given as in (\ref{resultant}). In particular every irreducible component is isomorphic to a plane conic and  is smooth.	
\end{theorem}

Clearly, the linear polynomials $L_2$ and $L_{n-1}$ in Theorem \ref{smoothd} are given by $L_2 = x_2-x_1$, $L_{n-1} = x_{n-1}-x_n$.

\subsection{Type $\widetilde{\mathbb E}_6$}
In this subsection, we will give the Gr\"obner bases  for  irreducible components of the frieze variety  $X(Q)$ of type $\widetilde{\mathbb E}_6$.  Recall in this case, the period $m=6$. We are going to give the first two irreducible components  $X_0$ and $X_1$ of $X(Q)$, and we will show that $X(Q) = X_0\cup X_1$.  

Suppose that 
$ \mathbf a_j=\mathbf x_{j}(\mathbf 1)$ for $j\in\mathbb N$.
The  image of the projection $\pi:\mathbb C^7\rightarrow \mathbb C^3$ mapping each $\mathbf x\in \mathbb C^n$ to $\pi(\mathbf x) = (x_1,x_2,x_7)$ of   first  thirteen points $\{\mathbf a_j, 0\leq j\leq 12\}$  are given by 
\[
\begin{split}	
	&(1,1,1),\,(2,3,28),\,(2,19,245),\,(10,129,8762),\, (13,883,78574),\\
	&(68,6051,2819698),
	(89,41473,25298441),\,(466,284259,907922780),\\
	&(610,1948339,8146004749),\,\,(3194, 13354113, 292348238602),\\
	&(4181,91530451, 2622988130126),\,\,(21892, 627359043,94135224380258),\\
	&(28657, 4299982849,844594031206225)\dots
\end{split}\]

The sequence $(a_{1,j})_{j\in\mathbb N}$  satisfies a linear recurrence with characteristic polynomial \[  x^{12} - 322x^6+1.\]

It is easy to check that  $a_{1,j+4} - 7a_{1,j+2}+a_{1,j} =0$ holds for $0\leq j\leq 11$. Then by 
[Lemma 3.13, \cite{lee2020frieze}],  the sequence $(a_{1,j})_{j\in\mathbb N}$   satisfies a linear sequence with the characteristic polynomial 
\[   x^4  -  7x^2+1. \]

Thus  for any  $k\geq 0$, the sub-sequence $(a_{1,k+2j})_{j\in \mathbb N}$ satisfies a linear recurrence with characteristic polynomial \[p(x) =x^2-7x+1 = (x-\rho_0)(x-\rho_1),\]
where  \[  \rho_0 = \frac{7-3\sqrt{5}}{2}, \quad \text{and}\;\, \rho_1=\frac{7+3\sqrt{5}}{2}.\]

This implies that there are at most two irreducible components of the frieze variety $X(Q)$.

It is also obvious that for $k\geq 0$, we have 
\[r_{1,k}(t)= r_{3,k}(t)= r_{5,k}(t), \quad r_{2,k}(t)= r_{4,k}(t)= r_{6,k}(t).\]

Thus we only need to find rational functions $r_{1,k}(t), r_{2,k}(t)$ and $r_{7,k}(t)$.
Moreover, to give $r_{2, k}(t)$ and $r_{7,k}(t)$,  it  is enough to compute  $r_{1,k}(t), r_{1,k+1}(t), r_{1,k+2}(t)$ and  use the following relations:
\begin{equation}\label{27}
	\begin{split}  r_{2,k}(t) &= r_{1,k}(t)r_{1,k+1}(t)  -1,\\
		r_{7,k}(t) &=  r_{1,k}(t)r_{2,k+1}(t)-r_{1,k+2}(t) \\
		&=  r_{1,k}(t)(r_{1,k+1}(t)r_{1,k+2}(t)-1)-r_{1,k+2}(t). \end{split}	\end{equation}

Therefore for $k=0,1$, it is enough to compute 
$r_{1,i}(t)$ for $0\leq i\leq 3$ and other rational functions we need are combinations of them.
Since $r_{1,i}(t)$ is of the form $m_{0i}t^{-1} + m_{1i}t$ for $0\leq i\leq 3$,
it suffices to solve the following linear equations:
\begin{equation}\label{27equation} r_{1,i}(\rho^s) = a_{1,  i+2s};\quad s=0,1;\,\, i=0,1,2,3.\end{equation}

More precisely, we shall solve the following linear equations:
\begin{equation}\label{27matrix}
	\begin{pmatrix}   1 &1\\  \rho_0&\rho_1\end{pmatrix} \cdot \begin{pmatrix} m_{00}&m_{01}&m_{02}&m_{03}\\m_{10}&m_{11}&m_{12} &m_{13} \end{pmatrix}= \begin{pmatrix} 1&2&2&10\\ 2& 10&  13&68 \end{pmatrix}.
\end{equation}

Combining (\ref{27}), (\ref{27equation}) and (\ref{27matrix}),
for $k=0$, we have
\[   \begin{split} 
	r_{1,0}(t) &=r_{3,0}(t) =r_{5,0}(t) = \frac{(5-\sqrt{5})t^2+5+\sqrt{5} }{10t},\\
	r_{2,0}(t) & = r_{4,0}(t) =r_{6,0}(t) =\frac{2t^4+t^2+2}{5t^2},\\
	r_{7,0}(t) &=\frac{(20+8\sqrt{5})t^6+(5+\sqrt{5})t^4+(5-\sqrt{5})t^2+20 -8\sqrt{5}}{50t^3}.
\end{split}                   \]
and for $k=1$, we have
\[
\begin{split}
	r_{1,1}(t) &= r_{3,1}(t) = r_{5,1}(t) = \frac{(5+\sqrt{5})t^2+5-\sqrt{5}}{5t},\\
	r_{2,1}(t) &= r_{4,1}(t) = r_{6,1}(t) =\frac{(7+3\sqrt{5})t^4+t^2+7-3\sqrt{5}}{5t^2},\\
	r_{7,1}(t) &=\frac{(340+152\sqrt{5})t^6+(10+4\sqrt{5})t^4+(10-4\sqrt{5})t^2+340-
		152\sqrt{5} }{25t^3}. 
\end{split}
\]

The Gr{\"o}bner basis for $\I(X_0)$ with respect to the order $x_7>\dots>x_2>x_1$ is given by the following polynomials: 
\begin{equation}\label{e60}
	\begin{split}h_0 &= 2x_1^2(2x_1^2-3x_2-1)+(x_2+1)^2,\\
		h_3 &= x_3 - x_1,\quad h_4= x_4-x_2,\quad h_5= x_5-x_1,\quad h_6=x_6-x_2,\\
		h_1&=4x_1^2(3x_1^2 -4x_2-1) +3(x_2+1)+   2x_1x_7,\\
		h_2&=x_1^3(32x_1^2-42x_2-16)+x_1(15x_2+9)+x_7(x_2+1),\\
		h_7&=x_1^4(440x_1^2-576x_2-400)+x_1^2(444x_2+198)-63x_2+2x_7^2-45.
\end{split}\end{equation}
Thus  $\I(X_0) = \langle h_i, 0\leq i\leq 7\rangle$.  

The Gr{\"o}bner basis for $\I(X_1)$ with respect to the order $x_7>\dots>x_2>x_1$  is given by the following  polynomials: 
\begin{equation}\label{e61}
	\begin{split}
		g_0 &=x_1^2(x_1^2-6x_2-2)+4(x_2+1)^2,\\
		g_3 &=x_3-x_1,\quad g_4=x_4-x_2,\quad g_5=x_5-x_1,\quad g_6=x_6-x_2,\\
		g_1&=x_1^2(3x_1^2-16x_2-4)+12(x_2+1)+2x_1x_7,\\
		g_2&=x_1^3(4x_1^2-21x_2-8)+x_1(30x_2+18)+2x_7(x_2+1),\\
		g_7&= x_1^4(55x_1^2-288x_2-200)+x_1^2(888x_2+396) -504x_2+4x_7^2-360.
	\end{split}
\end{equation}
Then $\I(X_1) = \langle g_i, 0\leq i\leq 7\rangle$.  

\begin{theorem}
	Let $Q$ be the affine quiver of type $\tilde{\mathbb E}_6$.  Then the frieze variety $X(Q) = X_0\cup X_1$ is a union of two rational curves.  Furthermore, the Gr\"obner bases for $X_0$ and  $X_1$ with respect to the order $x_7>\dots>x_2>x_1$ are polynomials given in (\ref{e60}) and (\ref{e61}) respectively.  In particular, $X_0$ and $X_1$ are smooth curves.
\end{theorem}
\begin{proof}
	We have already gave the Gr\"obner bases for $X_0$ and $X_1$.  
	For the smoothness, it is not hard to show by finding a nonsingular $2\times 2$ submatrix
	of the Jacobian matrices 
	\[J_{p}=(\frac{\partial h_j}{\partial x_i}(p))_{i=1,2,7;j=0,1,2,7},\quad \text{and}\;\, J_q=(\frac{\partial g_i}{\partial x_j}(q))_{i=1,2,7;j=0,1,2,7},\]
	for each $p\in X_0$ and $q\in X_1$.
\end{proof}

\subsection{Type $\widetilde{\mathbb E}_7$}
The sequence $(a_{1,j})_{j\in \mathbb N}$ is given by 
\[
	1,2,2,2,9,13,13,79,89,115,544,788,792,4817,5427,7013,33175\dots
	\]
The sequence $(a_{1,j})_{j\in\mathbb N}$  satisfies a linear recurrence with characteristic polynomial \[  x^{24} - 3719x^{12}+1.\]

It is easy to check that  $a_{1,j+12} - 61a_{1, j+6}+a_{1,j} =0$ holds for $0\leq j\leq 23$. Then by  [Lemma 3.13, \cite{lee2020frieze}],  the sequence $(a_{1,j})_{j\in\mathbb N}$   satisfies a linear recurrence with the characteristic polynomial 
\[   p(x)=x^{12}  -  61x^6 +1. \]

Thus  for any  $k\geq 0$, the sub-sequence $(a_{1,k+6j})_{j\in \mathbb N}$ satisfies a linear recurrence with characteristic polynomial $p(x) =x^2-61x+1$.
We have that
\[	r_{1,-1}(t) = \frac{(2065-93\sqrt{413})t^2+2065+93\sqrt{413}}{826t},\]
\[	r_{1,0}(t) = \frac{(177-5\sqrt{413})t^2+177+5\sqrt{413}}{354t},\]
\[	r_{1,1}(t) = \frac{(413+6\sqrt{413})t^2+413-6\sqrt{413}}{413t},\]
\[	r_{1,2}(t) = \frac{(177+4\sqrt{413})t^2+177-4\sqrt{413}}{177t},\]
\[	r_{1,3}(t) = \frac{(413+18\sqrt{413})t^2+413-18\sqrt{413}}{413t}.\]

The other rational functions $r_{j,0}(t)$'s are combinations of the above five rational functions. Indeed, they are given by 
\[
\begin{split}
	r_{2,0} &=r_{1,0}r_{1,1}-1,\quad r_{3,0} = r_{1,2}r_{2,0} -r_{1,0},\quad r_{7,0} = r_{1,-1}r_{1,3}-9,\\
	r_{8,0} &= r_{1,0}r_{3,1}-r_{2,2}=r_{2,0}(r_{1,2}r_{1,3}-1)-r_{1,0}r_{1,3},\\
	r_{4,0}&=r_{1,0},\quad r_{5,0}=r_{2,0},\quad r_{6,0} = r_{3,0}.
\end{split}
\]

The Gr\"obner basis for the first irreducible component $X_0$ with respect to the lexicographical order  $x_8>\dots>x_2>x_1$ is given by 
\[
\begin{split}
	h_0 =& x_1^2(3393x_1^2-4347x_2-1538)+623(x_2+1)^2,\\
	h_1= &x_1^2(30537x_1^2 -37343x_2-9125)+5607(x_2+1)+4717x_1x_3,\\
	h_2=&x_1^3(126704799x_1^2-143305470x_2-57433534)+x_1(41953443x_2+26203380)\\&+2938691x_3(x_2+1),\\
	h_3=& x_1^4(78312916890x_1^2-87566459589x_2-60724717366)+x_1^2(55094217528x_2+\\&25552599807)-6298169283x_2+261543499x_3^2   -4631931486,\\
	h_4=&x_4-x_1,\quad h_5 = x_5-x_2,\quad h_6=x_6-x_3,\\
	h_7=&180x_1^2-567x_2+371x_7+16,\\
	h_8=&x_1^2(2100267x_1^2-2348766x_2-996878)+771274x_2+33019x_8+441084.
\end{split}
\]

Thus  $\I(X_0) = \langle h_i, 0\leq i\leq 8\rangle$.  Notice that $h_{i}$ is linear on $x_i$ for $4\leq i\leq 8$, and the plane curve $\Z(h_0(x,y))\subset \mathbb C^2$ has only one nonsingular point $(0,-1)$. 

To see that $X_0$ is smooth, it is enough to  find a nonsingular $2\times 2$ submatrix of  the Jacobian matrix at each vertex $p\in X_0$:
 \[J_{p} =  (\frac{\partial h_j}{\partial x_i}(p))_{1\leq i\leq 3,0\leq j\leq 3}.\]

  Suppose that $p=(p_1,\dots,p_8)\in X_0$.   If $(p_1,p_2)\neq (0,-1)$,  it is  clear that $p$ is nonsingular.  
   If $(p_1,p_2) = (0,-1)$, then $p_3^2= -567/89$ and it is also easy to check that  $p$ is nonsingular.  
   
 The Gr\"obner bases for the other $5$ irreducible components are obtained similarly and also have similar forms as above.   We get the following result.

\begin{theorem}
	Let $Q$ be the affine quiver of type $\tilde{\mathbb E}_7$. Then the frieze variety $X(Q)$ has exactly $6$ irreducible components, each of which is a smooth rational curve.	
\end{theorem}

\subsection{Type $\widetilde{\mathbb E}_8$}
In this case, the characteristic polynomial for the frieze sequence associated to the extending vertex $1$ is given by 
\[  p(x) =  x^{60}-403520x^{30}+1.        \]
 
 There are exactly $30$ irreducible components,  and we are not going to list all of Gr\"obner bases of these components.  Instead we give the Gr\"obner basis of the first component with respect to the lexicographical order $x_9>\dots>x_1$ as follows, and the others are of similar forms. 

{\small 
	\begin{align*} 
		h_0=&x_1^{2}(22543305725 x_1^{2}-26141135142  x_2-8110066079)+2926973874( x_2+1)^2, \\
		h_1=&x_1^{2}(16343896650625 x_1^{2}-18723148834356  x_2-3872328920422)
		\\&+2007468986853 x_1 x_3 +2122056058650 (x_2+1),\\
		h_2=&x_1^3(70346944717927448581350 x_1^{2}-73601034244490747961217 x_2 \\ &-25307662375596497568354)+16490152683938879106436 x_1x_2\\ &+10112996125759591015611 x_1+979301546230663413087x_3(x_2+1), \\
		h_{3}=&x_1^{4}(23383957524892770437902768648325x_1^{2}
		-24210646824344145750432418618584  x_2
		\\&-14002149395163331739829469440508) +11499945776245224835235334136644 x_1^{2} x_2
		\\&+5008293274220414979609114513086 x_1^{2}+38735000548445337672882684987 x_3^{2}
		\\&-992386730386199596313604245100 x_2-725748626013178503844607678850,\\
		h_4=&x_1^2(359194242257305722650 x_1^{2}-374728683921258002475x_2
		\\&   -130558886700336479750) +4483394914071055521 x_4
		\\&+89821049942925712125 x_2+51788883507291991929,\\
		h_5=&x_1^3(34269327083645850969959832434375 x_1^{2}-
		35455633995008172262734155417500 x_2
		\\&-12821905410327630628329344566250)+8431648068155826092943294521602 x_1x_2
		\\&+42855924958221078851074695075 x_5
		+5036048458218815748413466112899 x_1
		\\&+497659870357089000895832219799 x_3,\\
		h_6=&x_1(-83816240368052250 x_1^{2}+208895750256064746x_2+7746683096103102)
		\\&-258477022652889623 x_3+125650829668774025 x_6,\\
		h_7=&1036864580 x_1^{2}-2701572585 x_2+1629628077 x_7+35079928,     \\
		h_8=& x_1^{2}(782877279783238984727450 x_1^{2}
		-814535978663555483485075 x_2
		\\&-291737428469833096141830)+203185971689938628179129 x_2
		\\&+6855110823614643891609 x_8+113355044836596322828717,\\
		h_{9}=&x_1^4(2726538880152109668781259112672250 x_1^{2}-2819349208031363391362477411375475 x_2
		\\&-1702791311300737001077661419569500 )+609132147201086599311161270518771 x_1^{2}
		\\&+x_2(1412169293224665453709653706876878 x_1^{2}-132300105635140811563888457586822 )
		\\&+271145003839117363710178794909 x_9-93670840614459635161756980331011.
\end{align*}}

Notice that $h_i$ is linear on $x_i$ for $4\leq i\leq 9$. Similar computations  as above for the cases $\tilde{\mathbb E}_6$ and $\tilde{\mathbb E}_7$, we have the following result.
\begin{theorem}
Let $Q$ be the affine quiver of type $\tilde{\mathbb E}_8$. Then the frieze variety $X(Q)$ has exactly $30$ irreducible components, each of which is a smooth rational curve.	
\end{theorem}

\begin{remark}
	If we break the symmetry for types $\tilde{\mathbb E}_6$ and $\tilde{\mathbb E}_7$, the number of irreducible components could be maximal, \ie $6$ for type $\tilde{\mathbb E}_6$ and $12$ for type $\tilde{\mathbb E}_7$. For example, let $\mathbf a= (1,1,3,2,2,1,1)$ for  $\tilde{\mathbb E}_6$ and  let $\mathbf a= (1,1,1,2,1,1,1,1)$ for  $\tilde{\mathbb E}_7$.  In this case, the characteristic polynomials for extending vertices are given by  $x^{12} -250x^6+1$ and $x^{24} - 4990x^{12}+1$  for types $\tilde{\mathbb E}_6$ and $\tilde{\mathbb E}_7$, respectively.
	
	Of independent interest, we also note that  $x^{12} -250x^6+1$,  $x^{24} - 4990x^{12}+1$ and $x^{60}-403520x^{30}+1$  are irreducible polynomials in $\mathbb Q[x]$, which imply that they are minimal characteristic polynomials for the frieze sequence  $(a_{i,j})_{j\in \mathbb N}$ (over $\mathbb Q$) associated with any extending vertex $i$ of types $\tilde{\mathbb E}_6$, $\tilde{\mathbb E}_7$  and $\tilde{\mathbb E}_8$ respectively. This  implies that for affine quivers of exceptional types, the polynomial  $x^{2m} - \mathcal L x^m+1$ is the minimal characteristic polynomial of  the sequence $(x_{i,j})_{j\in \mathbb N}$ of rational functions in the field $\mathbb Q(\mathbf x)$ for each extending vertex $i$, where  $m$ and $\mathcal L$ are given in Table \ref{tab1} for exceptional types.  
	\end{remark}

\subsection{Non-simply laced cases}
Recall that every non-simply laced  affine valued quiver is a  folding  of   simply laced affine quiver  $Q$ with respect to a  group $\Gamma$, see \cite{lutztig1991quantum} for details.  By Theorem \ref{folding},  we can obtain Gr\"obner bases  for non-simply laced affine valued quivers.

\subsection*{Acknowledgment}  The author would like to thank Professor Fang Li  for his continuing  encouragement.  This project is supported by the National Natural Science Foundation of China (No. 12071422).

\def\cprime{$'$} \def\cprime{$'$}
\providecommand{\bysame}{\leavevmode\hbox to3em{\hrulefill}\thinspace}
\providecommand{\MR}{\relax\ifhmode\unskip\space\fi MR }
\providecommand{\MRhref}[2]{%
	\href{http://www.ams.org/mathscinet-getitem?mr=#1}{#2}
}
\providecommand{\href}[2]{#2}

\end{document}